\pdfoutput=1
\documentclass[12pt]{article}

\pdftrailerid{}

\usepackage[T1]{fontenc}
\usepackage[english]{babel}
\usepackage{lmodern}

\usepackage{amsmath}
\usepackage{amsthm}
\usepackage[bbgreekl]{mathbbol} 
\usepackage{amssymb}
\DeclareSymbolFontAlphabet{\mathbb}{AMSb}
\DeclareSymbolFontAlphabet{\mathbbl}{bbold}
\usepackage{mathrsfs}
\usepackage{bm}
\usepackage{bbm}
\usepackage{mathtools}
\usepackage{thmtools}

\usepackage[final,babel]{microtype}
\usepackage{csquotes}

\usepackage[
backend=biber,
sorting=nyt,
style=alphabetic,
backref=true,
date=year,
useprefix=true,
maxnames=9,
labelalpha,
maxalphanames=9
]{biblatex}

\usepackage{shortmathj}
\DeclareFieldFormat[article]{titlecase}{\shortifyAMSjournalname{#1}}
\DeclareFieldFormat{postnote}{#1}
\DeclareFieldFormat{pages}{#1}
\DeclareFieldFormat[article, incollection]{title}{#1}
\DeclareFieldFormat{urldate}{(#1)}
\DeclareFieldFormat[article]{volume}{\textbf{#1}\nopunct}
\DeclareFieldFormat[article]{number}{(\textbf{#1})}
\DefineBibliographyStrings{english}{
	backrefpage={},
	backrefpages={},
}
\usepackage{xpatch}
\DeclareFieldFormat{backrefparens}{\addperiod\allowbreak#1\addperiod}
\xpatchbibmacro{pageref}{parens}{backrefparens}{}{}
\renewbibmacro{in:}{%
	\ifentrytype{article}{}{\printtext{\bibstring{in} }}}

\usepackage[useregional]{datetime2}
\usepackage{geometry}
\usepackage[table]{xcolor}
\definecolor{HRlink}{HTML}{800006}
\definecolor{HRcite}{HTML}{2E7E2A}
\definecolor{HRfile}{HTML}{131877}
\definecolor{HRurl}{HTML}{8A0087}
\definecolor{HRmenu}{HTML}{727500}
\definecolor{HRrun}{HTML}{137776}
\usepackage{hyperref}
\hypersetup{
	hypertexnames=false,
	unicode,
	colorlinks,
	linkcolor=HRlink,
	citecolor=HRcite,
	filecolor=HRfile,
	urlcolor=HRurl,
	menucolor=HRmenu,
	runcolor=HRrun,
}
\usepackage[nameinlink,capitalise]{cleveref}
\crefname{equation}{}{} 

\usepackage{soul}
\usepackage{enumitem}
\usepackage{cancel}

\declaretheorem{thm}[
numberwithin=section,
name=Theorem,
refname={Theorem,Theorems},
Refname={Theorem,Theorems},
]
\declaretheorem{lem}[
sharenumber=thm,
name=Lemma,
refname={Lemma,Lemmas},
Refname={Lemma,Lemmas},
]
\declaretheorem{cor}[
sharenumber=thm,
name=Corollary,
refname={Corollary,Corollaries},
Refname={Corollary,Corollaries},
]
\declaretheorem{prop}[
sharenumber=thm,
name=Proposition,
refname={Proposition,Propositions},
Refname={Proposition,Propositions},
]
\declaretheorem{conj}[
sharenumber=thm,
name=Conjecture,
refname={Conjecture,Conjectures},
Refname={Conjecture,Conjectures},
]

\declaretheorem{claim}[
sharenumber=thm,
style=definition,
name=Claim,
refname={Claim,Claims},
Refname={Claim,Claims},
]
[
sharenumber=thm,
style=definition,
name=Definition,
refname={Definition,Definitions},
Refname={Definition,Definitions},
]
\declaretheorem{eg}[
sharenumber=thm,
style=definition,
name=Example,
refname={Example,Examples},
Refname={Example,Examples},
]
\declaretheorem{rmk}[
sharenumber=thm,
style=definition,
name=Remark,
refname={Remark,Remarks},
Refname={Remark,Remarks},
]

\declaretheorem{prob}[
sharenumber=thm,
style=definition,
name=Problem,
refname={Problem,Problems},
Refname={Problem,Problems},
]

\newenvironment{claimproof}[1][\proofname]
{%
	\proof[#1]%
}
{%
	\endproof%
}

\newcommand{\tb}{\textbf}
\newcommand{\bb}[1]{\mathbb{#1}}
\newcommand{\res}[2]{#1 {\upharpoonright} #2}

\newcommand{\N}{\bb{N}}
\newcommand{\Z}{\bb{Z}}
\newcommand{\R}{\bb{R}}
\newcommand{\bbD}{\bb{D}}
\newcommand{\bbF}{\bb{F}}
\newcommand{\bbH}{\bb{H}}

\DeclareMathOperator{\proj}{proj}
\DeclareMathOperator{\graph}{graph}

\let\hom\relax
\DeclareMathOperator{\hom}{Hom}
\DeclareMathOperator{\diam}{diam}
\DeclareMathOperator{\FBU}{FBU}

\let\H\relax
\newcommand{\H}{\mathcal{H}}
\newcommand{\F}{\mathcal{F}}
\newcommand{\calP}{\mathcal{P}}

\newcommand{\E}{\bb{E}}
\newcommand{\G}{\bb{G}}

\begin{document}
\title{Invariant uniformization}
\author{Alexander S. Kechris and Michael Wolman}
\date{}
\maketitle

\begin{center}
	Dedicated to Menachem Magidor
\end{center}

\vspace{1pt}

\begin{abstract}
	Standard results in descriptive set theory provide sufficient conditions for a Borel set $P \subseteq \N^\N \times \N^\N$ to admit a Borel uniformization, namely, when $P$ has ``small'' sections or ``large'' sections. We consider an invariant analogue of these results: Given a Borel equivalence relation $E$ and an $E$-invariant Borel set $P$ with ``small'' or ``large'' sections, does $P$ admit an $E$-invariant Borel uniformization?

	For a given Borel equivalence relation $E$, we show that every $E$-invariant Borel set $P$ with ``small'' or ``large'' sections admits an $E$-invariant Borel uniformization if and only if $E$ is smooth. We also compute the definable complexity of counterexamples in the case where $E$ is not smooth, using category, measure, and Ramsey-theoretic methods.

	We provide two new proofs of a dichotomy of Miller classifying the pairs $(E, P)$ such that $P$ admits an $E$-invariant uniformization, for a Borel equivalence relation $E$ and a Borel $E$-invariant set $P$ with countable sections. In the process, we prove an $\aleph_0$-dimensional $(\G_0, \bbH_0)$ dichotomy, generalizing dichotomies of Miller and Lecomte. We also show that the set of pairs $(E, P)$ such that $P$ has ``large'' sections and admits an $E$-invariant Borel uniformization is $\bm{\Sigma^1_2}$-complete; in particular, there is no analog of Miller's dichotomy for $P$ with ``large'' sections.

	Finally, we consider a less strict notion of invariant uniformization, where we select a countable nonempty subset of each section instead of a single point.
\end{abstract}

\section{Introduction}

\subsection{Invariant uniformization and smoothness}

Given sets $X,Y$ and $P\subseteq X\times Y$ with $\proj_X (P) =X$, a \tb{uniformization} of $P$ is a function $f\colon X\to Y$ such that $\forall x \in X ((x,f(x))\in P)$. If now $E$ is an equivalence relation on $X$, we say that $P$ is \tb{$\boldsymbol{E}$-invariant} if $x_1 E x_2 \implies P_{x_1} = P_{x_2}$, where $P_x = \{y\colon (x,y) \in P\}$ is the $x$-\tb{section} of $P$. Equivalently this means that $P$ is invariant under the equivalence relation $E\times \Delta_Y$ on $X\times Y$, where $\Delta_Y$ is the equality relation on $Y$. In this case an \tb{$\boldsymbol{E}$-invariant uniformization} is a uniformization $f$ such that $x_1 E x_2 \implies f(x_1) = f(x_2)$.

If $E, F$ are equivalence relations on sets $X,Y$, resp., a \tb{homomorphism} of $E$ to $F$ is a function $f\colon X\to Y$ such that $x_1 E x_2 \implies f(x_1) F f(x_2)$. Thus an invariant uniformization is a uniformization that is a homomorphism of $E$ to $\Delta_Y$.

Consider now the situation where $X,Y$ are Polish spaces and $P$ is a Borel subset of $X\times Y$. In this case standard results in descriptive set theory provide conditions which imply the existence of Borel uniformizations. These fall mainly into two categories, see \cite[Section 18]{CDST}: ``small section'' and ``large section'' uniformization results. We will concentrate here on the following standard instances of these results:

\begin{thm}[Measure uniformization]\label{thm:measure-uniformization}
Let $X,Y$ be Polish spaces, $\mu$ a probability Borel measure on $Y$ and $P\subseteq X\times Y$ a Borel set such that $\forall x \in X (\mu (P_x) > 0)$. Then $P$ admits a Borel uniformization.
\end{thm}

\begin{thm}[Category uniformization]\label{thm:category-uniformization}
Let $X,Y$ be Polish spaces and $P\subseteq X\times Y$ a Borel set such that $\forall x \in X (P_x \textrm{ is non-meager})$. Then $P$ admits a Borel uniformization.
\end{thm}

\begin{thm}[$K_\sigma$ uniformization]\label{thm:K-sigma-uniformization}
Let $X,Y$ be Polish spaces and $P\subseteq X\times Y$ a Borel set such that $\forall x \in X (P_x \textrm{ is non-empty and } K_\sigma)$. Then $P$ admits a Borel uniformization.
\end{thm}

A special case of \cref{thm:K-sigma-uniformization} is the following:

\begin{thm}[Countable uniformization]\label{thm:countable-uniformization}
Let $X,Y$ be Polish spaces and $P\subseteq X\times Y$ a Borel set such that $\forall x \in X (P_x \textrm{ is non empty and countable})$. Then $P$ admits a Borel uniformization.
\end{thm}

Suppose now that $E$ is a Borel equivalence relation on $X$ and $P$ in any one of these results is $E$-invariant. When does there exist a \tb{Borel $\boldsymbol{E}$-invariant uniformization}, i.e., a Borel uniformization that is also a homomorphism of $E$ to $\Delta_Y$? We say that $E$ satisfies \tb{measure (resp., category, $\boldsymbol{K_\sigma}$, countable) invariant uniformization} if for every $Y, \mu, P$ as in the corresponding uniformization theorem above, if $P$ is moreover $E$-invariant, then it admits a Borel $E$-invariant uniformization.

The following gives a complete answer to this question. Recall that a Borel equivalence relation $E$ on $X$ is \tb{smooth} if there is a Polish space $Z$ and a Borel function $S\colon X\to Z$ such that $x_1 E x_2 \iff S(x_1) = S(x_2)$.

\begin{thm}\label{thm:equivalence-uniformization-smooth}
Let $E$ be a Borel equivalence relation on a Polish space $X$. Then the following are equivalent:
\begin{enumerate}[label=(\roman*)]
	\item $E$ is smooth;

	\item $E$ satisfies measure invariant uniformization;

	\item $E$ satisfies  category invariant uniformization;

	\item $E$ satisfies $K_\sigma$ invariant uniformization;

	\item $E$ satisfies countable invariant uniformization.
\end{enumerate}
\end{thm}

The proof is given in \cref{subsec:smooth-case,subsec:non-smooth-case}.

One can compute the exact definable complexity of counterexamples to invariant uniformization. Let $\E_0$ denote the non-smooth Borel equivalence relation on $2^\N$ given by $x \E_0 y \iff \exists m \forall n\geq m (x_n = y_n)$. In the proof of \cref{thm:equivalence-uniformization-smooth}, it is shown that for $E = \E_0$ on $X = 2^\N$ we have the following:
\begin{enumerate}[label=(\arabic*)]
	\item Failure of measure invariant uniformization: There are $Y, \mu$, and an $E$-invariant $P \in F_\sigma$ with $\mu (P_x) =1$ for all $x\in X$, which has no Borel $E$-invariant uniformization.

	\item Failure of category invariant uniformization: There are $Y$ and an $E$-invariant $Q\in G_\delta$ with $Q_x$ comeager for all $x\in X$, which has no Borel $E$-invariant uniformization.

	\item Failure of countable invariant uniformization: There are $Y$ and an $E$-invariant $P\in F_\sigma$ such that $P_x$ is non-empty and countable for all $x\in X$, which has no Borel $E$-invariant uniformization.
\end{enumerate}

The definable complexity of $Q, P$ in (2), (3) is optimal. In the case of measure invariant uniformization, however, there are counterexamples which are $G_\delta$, and this together with (1) gives the optimal definable complexity of counterexamples to measure invariant uniformization. These results are the contents of \cref{thm:complexity-of-counterexamples,thm:ramsey-example}, which we prove in \cref{sec:complexity-counterexamples}.

\begin{thm}\label{thm:ramsey-example}
	Let $X \subseteq 2^\N$ be the set of sequences with infinitely many ones. There is a Polish space $Y$, a probability Borel measure $\mu$ on $Y$ and an $\E_0$-invariant $G_\delta$ set $P \subseteq X \times Y$ such that $P_x$ is comeager and $\mu(P_x) = 1$ for all $x \in X$, which has no Borel $\E_0$-invariant uniformization.
\end{thm}

\begin{thm}\label{thm:complexity-of-counterexamples}
	Let $X, Y$ be Polish spaces, $E$ a Borel equivalence relation on $X$ and $P \subseteq X \times Y$ an $E$-invariant Borel relation. Suppose one of the following holds:
	\begin{enumerate}[label=(\roman*)]
		\item $P_x \in \bm{\Delta^0_2}$ and $\mu_x(P_x) > 0$, for all $x \in X$ and some Borel assignment $x \mapsto \mu_x$ of probability Borel measures $\mu_x$ on $Y$.
		\item $P_x \in F_\sigma$ and $P_x$ non-meager for all $x \in X$.
		\item $P_x \in G_\delta$ and $P_x$ non-empty and $K_\sigma$ (in particular countable) for all $x \in X$.
	\end{enumerate}
	Then there is a Borel $E$-invariant uniformization.
\end{thm}

The proof of \cref{thm:ramsey-example} uses the Ramsey property.

\subsection{Local dichotomies}

The equivalence of (i) and (v) in \cref{thm:equivalence-uniformization-smooth} essentially reduces to the fact that if $E$ is a countable Borel equivalence relation (i.e., one for which all of its equivalence classes are countable) which is not smooth, then the relation
\[
(x,y) \in P \iff x Ey,
\]
is $E$-invariant with countable nonempty sections but has no $E$-invariant uniformization. Considering the problem of invariant uniformization ``locally'', Miller \cite{Mi-dichotomy} recently proved the following dichotomy that shows that this is essentially the only obstruction to (v). Below $\E_0\times I_\N$ is the equivalence relation on $2^\N \times \N$ given by
\[(x,m) \E_0\times I_\N (y,n) \iff x \E_0 y.\]
Also if $E,F$ are equivalence relations on spaces $X,Y$, resp., an \tb{embedding} of $E$ into $F$ is an injection $\pi\colon X \to Y$ such that $x_2  E x_2 \iff \pi (x_1) F \pi (x_2)$.

\begin{thm}[{\cite[Theorem 2]{Mi-dichotomy}}]\label{thm:ben-dichotomy}
Let $X,Y$ be Polish spaces, $E$ a Borel equivalence relation on $X$ and $P\subseteq X\times Y$ an $E$-invariant Borel relation  with countable non-empty sections. Then exactly of the following holds:
\begin{enumerate}[label=(\arabic*)]
	\item There is a Borel $E$-invariant uniformization,
	\item There is a continuous embedding $\pi_X\colon 2^\N \times \N \to X$ of $\E_0\times I_\N$ into $E$ and a continuous injection $\pi_Y\colon 2^\N \times \N \to Y$ such that for all $x, x' \in 2^\N \times \N$,
	\[\lnot (x \mathrel{\E_0 \times I_\N} x') \implies P_{\pi_X(x)} \cap P_{\pi_X(x')} = \emptyset\]
	and
	\[P_{\pi_X(x)} = \pi_Y([x]_{\E_0 \times I_\N}).\]
\end{enumerate}
\end{thm}

We provide in \cref{sec:proof-of-ben-1} a different proof of this dichotomy, using Miller's $(\G_0, \bbH_0)$ dichotomy \cite{Mi-survey} and Lecomte's $\aleph_0$-dimensional hypergraph dichotomy \cite{L}. Our proof relies on the following strengthening of $(i) \implies (v)$ of \cref{thm:equivalence-uniformization-smooth}, which is interesting in its own right:

\begin{thm}\label{thm:unif-from-smooth}
	Let $F$ be a smooth Borel equivalence relation on a Polish space $X$, $Y$ be a Polish space, and $P \subseteq X \times Y$ be a Borel set with countable sections. Suppose that
	\[\bigcap_{x \in C} P_x \neq \emptyset\]
	for every $F$-class $C$. Then $P$ admits a Borel $F$-invariant uniformization.
\end{thm}

In \cref{sec:aleph-0-dim-g0-h0} we prove an $\aleph_0$-dimensional $(\G_0, \bbH_0)$-type dichotomy, generalizing Lecomte's dichotomy in the same way that the $(\G_0, \bbH_0)$ dichotomy generalizes the $\G_0$ dichotomy, and use this to give still another proof of \cref{thm:ben-dichotomy} in \cref{sec:proof-of-ben-2}.

In the case of countable uniformization, the Lusin--Novikov Theorem asserts that $P$ can be covered by the graphs of countably-many Borel functions. When $E$ is smooth, the proof of \cref{thm:equivalence-uniformization-smooth} gives an invariant analogue of this fact (see \cref{thm:smooth-lusin-novikov}). De Rancourt and Miller \cite{dRM} have shown that $\E_0$ is essentially the only obstruction to invariant Lusin--Novikov:

\begin{thm}[{\cite[Theorem~4.11]{dRM}}]\label{thm:invariant-lusin-novikov}
	Let $X, Y$ be Polish spaces, $E$ a Borel equivalence relation on $X$ and $P \subseteq X \times Y$ an $E$-invariant Borel relation with countable non-empty sections. Then exactly one of the following holds:
\begin{enumerate}[label=(\arabic*)]
	\item There is a sequence $g_n: X \to Y$ of Borel $E$-invariant uniformizations with $P = \bigcup_n \graph(g_n)$.

	\item There is a continuous embedding $\pi_X: 2^\N \to X$ of $\E_0$ into $E$ and a continuous injection $\pi_Y: 2^\N \to Y$ such that for all $x \in 2^\N$, $P(\pi_X(x), \pi_Y(x))$.
\end{enumerate}
\end{thm}

We provide a different proof of this theorem in \cref{sec:invariant-lusin-novikov}, directly from Miller's $(\G_0, \bbH_0)$ dichotomy.

\subsection{Anti-dichotomy results}

Our next result can be viewed as a sort of anti-dichotomy theorem for large-section invariant uniformizations (see also the discussion in \cite[Section~1]{TV}). Informally, dichotomies such as \cref{thm:ben-dichotomy} provide upper bounds on the complexity of the collection of Borel sets satisfying certain combinatorial properties. Thus, one method of showing that there is no analogous dichotomy is to provide lower bounds on the complexity of such sets.

In order to state this precisely, we first fix a ``nice'' parametrization of the Borel relations on $\N^\N$, i.e., a $\bm{\Pi^1_1}$ set $\bbD \subseteq 2^\N$ and a map $\bbD \ni d \mapsto \bbD_d$ such that each $\bbD_d \subseteq \N^\N \times \N^\N, d \in \bbD$ is Borel, each Borel set in $\N^\N \times \N^\N$ appears as some $\bbD_d$, and so that these satisfy some natural definability properties (cf. \cite[Section~5]{AK}).

Define now
\[\calP = \{(d, e) : \text{$\bbD_d$ is an equivalence relation on $\N^\N$ and $\bbD_e$ is $\bbD_d$-invariant}\},\]
and let $\calP^{unif}$ denote the set of pairs $(d, e) \in \calP$ for which $\bbD_e$ admits a Borel $\bbD_d$-invariant uniformization. More generally, for any set $A$ of properties of sets $P \subseteq \N^\N \times \N^\N$, let $\calP_A$ (resp. $\calP_A^{unif}$) denote the set of pairs $(d, e)$ in $\calP$ (resp. $\calP^{unif}$) such that $\bbD_e$ satisfies all of the properties in $A$. Let $\calP_{ctble}$ (resp. $\calP_{ctble}^{unif}$) denote $\calP_A$ (resp. $\calP_A^{unif}$) for $A$ consisting of the property that $P$ has countable nonempty sections.

One can check that $\calP$ is $\bm{\Pi^1_1}$ and that $\calP^{unif}$ is $\bm{\Sigma^1_2}$. The same is true for $\calP_{ctble}$ and $\calP_{ctble}^{unif}$. In the latter case, however, the effective version of \cref{thm:ben-dichotomy} (see \cref{thm:effective-ben-dichotomy}) gives a better bound on the complexity:

\begin{prop}\label{prop:complexity-countable-sections}
	The set $\calP_{ctble}^{unif}$ is $\bm{\Pi^1_1}$.
\end{prop}

By contrast, in the case of large sections, we prove the following, where a set $B$ in a Polish space $X$ is called \tb{$\bm{\Sigma^1_2}$-complete} if it is $\bm{\Sigma^1_2}$, and for all zero-dimensional Polish spaces $Y$ and $\bm{\Sigma^1_2}$ sets $C \subseteq Y$ there is a continuous function $f: Y \to X$ such that $C = f^{-1}(B)$.

\begin{thm}\label{thm:complexity-large-sections}
	The set $\calP_A^{unif}$ is $\bm{\Sigma^1_2}$-complete, where $A$ is one of the following sets of properties of $P \subseteq \N^\N \times \N^\N$:
	\begin{enumerate}
		\item $P$ has non-meager sections;
		\item $P$ has non-meager $G_\delta$ sections;
		\item $P$ has non-meager sections and is $G_\delta$;
		\item $P$ has $\mu$-positive sections for some probability Borel measure $\mu$ on $\N^\N$;
		\item $P$ has $\mu$-positive $F_\sigma$ sections for some probability Borel measure $\mu$ on $\N^\N$;
		\item $P$ has $\mu$-positive sections for some probability Borel measure $\mu$ on $\N^\N$ and is $F_\sigma$.
	\end{enumerate}
	The same holds for comeager instead of non-meager, and $\mu$-conull instead of $\mu$-positive.

	In fact, there is a hyperfinite Borel equivalence relation $E$ with code $d \in \bbD$ such that for all such $A$ above, the set of $e \in \bbD$ such that $(d, e) \in \calP_A^{unif}$ is $\bm{\Sigma^1_2}$-complete.
\end{thm}

We prove \cref{prop:complexity-countable-sections,thm:complexity-large-sections} in \cref{sec:hardness-proofs}. On the other hand, the following is open:

\begin{prob}
	Is there an dichotomy or anti-dichotomy result, analogous to \cref{thm:ben-dichotomy} or \cref{thm:complexity-large-sections}, for the case where $P$ has $K_\sigma$ nonempty sections?
\end{prob}

While we do not know the answer to this problem, we note in \cref{sec:K-sigma-failure} that \cref{thm:unif-from-smooth} is false when the sections are only assumed to be $K_\sigma$:

\begin{prop}\label{prop:K-sigma-failure}
	There is a smooth countable Borel equivalence relation $F$ on $\N^\N$ and an open set $P \subseteq \N^\N \times 2^\N$ such that
	\[\bigcap_{x \in C} P_x \neq \emptyset\]
	for every $F$-class $C$, but which does not admit a Borel $F$-invariant uniformization.
\end{prop}

Note that such a set $P$ has open, and hence $K_\sigma$, section in $2^\N$.

\subsection{Invariant countable uniformization}

We next consider a somewhat less strict notion of invariant uniformization, where instead of selecting a single point in each section we select a countable nonempty subset. More precisely, given Polish spaces $X,Y$, a Borel equivalence relation $E$ on $X$ and an $E$-invariant  Borel set $P\subseteq X\times Y$, with $\proj_X (P) =X$, a Borel \tb{$\boldsymbol{E}$-invariant countable uniformization} is a Borel function $f\colon X \to Y^\N$ such that $\forall x\in X \forall n\in \N ( (x, f(x)_n) \in P)$ and
\[x_1 E x_2 \implies \{f(x_1)_n\colon n\in \N\} = \{f(x_2)_n\colon n\in \N\}.\]
Equivalently, if for each Polish space $Y$, we denote by $\E^Y_{ctble}$ the equivalence relation on $Y^\N$ given by
\[
(x_n) \E^Y_{ctble} (y_n) \iff \{x_n\colon n\in \N\} = \{y_n\colon n\in \N\},
\]
then an $E$-invariant countable uniformization is a Borel homomorphism $f$ of $E$ to $\E^Y_{ctble}$ such that for each $x,n$, we have that $(x,f(x)_n)\in P$.

We say that $E$ satisfies \tb{measure (resp., category, $\boldsymbol{K_\sigma}$) countable invariant uniformization} if for every $Y, \mu, P$ as in the corresponding uniformization theorem above, if $P$ is moreover $E$-invariant, then it admits a Borel $E$-invariant countable uniformization.

Recall that a Borel equivalence relation $E$ on $X$ is \tb{reducible to countable} if there is a Polish space $Z$, a countable Borel equivalence relation $F$ on $Z$ and a Borel function $S\colon X \to Z$ such that $x_1  E x_2 \iff S(x_1) F S(x_2)$.

As in the proof of the smooth part of \cref{thm:equivalence-uniformization-smooth} (see \cref{subsec:smooth-case}), one can see that if a Borel equivalence relation $E$ on $X$ is reducible to countable, then $E$ satisfies measure (resp. category, $K_\sigma$) countable invariant uniformization. We conjecture the following:

\begin{conj}\label{conj:countable-uniformization}
Let $E$ be a Borel equivalence relation on a Polish space $X$. Then the following are equivalent:
\begin{enumerate}[label=(\alph*)]
	\item $E$ is reducible to countable;
	\item $E$ satisfies measure countable invariant uniformization;
	\item $E$ satisfies category countable invariant uniformization;
	\item $E$ satisfies $K_\sigma$ countable invariant uniformization.
\end{enumerate}
\end{conj}
We discuss some partial results in \cref{sec:countable-uniformization}.

\subsection{Further invariant uniformization results and smoothness}

We have so far considered the existence of Borel invariant uniformizations, generalizing the standard ``small section'' and ``large section'' uniformization theorems. One can also consider invariant analogues of uniformization theorems for more general pointclasses, such as the following:

\begin{thm}[Jankov, von Neumann uniformization {\cite[18.1]{CDST}}]
	Let $X, Y$ be Polish spaces and $P \subseteq X \times Y$ be a $\bm{\Sigma^1_1}$ set such that $P_x$ is non-empty for all $x \in X$. Then $P$ has a uniformization function which is $\sigma(\bm{\Sigma^1_1})$-measurable.
\end{thm}

\begin{thm}[Novikov--Kond\^o uniformization {\cite[36.14]{CDST}}]
	Let $X, Y$ be Polish spaces and $P \subseteq X \times Y$ be a $\bm{\Pi^1_1}$ set such that $P_x$ is non-empty for all $x \in X$. Then $P$ has a uniformization function whose graph is $\bm{\Pi^1_1}$.
\end{thm}

Let $E$ be a Borel equivalence relation on $X$. We say $E$ satisfies \textbf{Jankov--von Neumann (resp. Novikov--Kond\^o) invariant uniformization} if for every $Y, P$ as in the corresponding uniformization theorem above, if $P$ is moreover $E$-invariant, then it admits an $E$-invariant uniformization which is definable in the same sense as in the corresponding uniformization theorem.

The following characterization of those Borel equivalence relations that satisfy these properties essentially follows from the proof of \cref{thm:equivalence-uniformization-smooth}, and is proved in \cref{subsec:proof-equiv-other-unif}.

\begin{thm}\label{thm:equiv-unif-smooth-jvn-nk}
	Let $E$ be a Borel equivalence relation on a Polish space $X$. Then the following are equivalent:
	\begin{enumerate}[label=(\roman*)]
		\item $E$ is smooth;
		\item $E$ satisfies Jankov--von Neumann invariant uniformization;
		\item $E$ satisfies Novikov--Kond\^o invariant uniformization.
	\end{enumerate}
\end{thm}

\subsection{Remarks on invariant uniformization over products}

One can consider more generally the question of invariant uniformization over products. Let $X, Y$ be Polish spaces, $E$ a Borel equivalence on $X$, $F$ a Borel equivalence on $Y$, and $P \subseteq X \times Y$ an $E \times F$-invariant set. In this case, one can ask whether there is an $E \times F$-invariant Borel set $U \subseteq P$ so that each section $U_x$ intersects one, or even finitely-many, $F$-classes. This paper then considers the special case where $F = \Delta_Y$ is equality.

In the case where $P$ has countable sections and $F$ is smooth, one can reduce this to the case where $F$ is equality to get analogues of \cref{thm:ben-dichotomy,thm:invariant-lusin-novikov}.

Miller \cite[Theorem~2.1]{Mi-dichotomy} has proved a generalization of \cref{thm:ben-dichotomy} where $P$ has countable sections and the equivalence classes of $F$ are countable, and de Rancourt and Miller \cite[Theorem~4.11]{dRM} have proved a generalization of \cref{thm:invariant-lusin-novikov} where the sections of $P$ are contained in countably many $F$-classes (but are not necessarily countable). See \cite[Chapter~4]{Wolman-thesis} for a survey of these results, as well as more dichotomies regarding invariant uniformizations over products.

The problem of invariant uniformization is also discussed in \cite{myers_invariant_1976,burgess_remarks_1975} where they consider the question of invariant uniformization over products when $E, F$ come from Polish group actions, and specifically when $E, F$ are the isomorphism relation on a class of structures. Myers \cite[Theorem~10]{myers_invariant_1976} gives an example in which there is no Baire-measurable invariant uniformization, so that in particular the invariant Jankov--von Neumann and invariant Novikov--Kond\^o uniformization don't hold.

\begin{rmk}
	We have defined an $E \times F$-invariant uniformizing \emph{set} for $P$ to be an $E \times F$-invariant set $U \subseteq P$ so that each section $U_x$ contains exactly one $F$-class. It is also natural to consider the (a priori stronger) notion of an $E \times F$-invariant uniformizing \emph{function} for $P$, i.e., a map $f: X \to Y$ whose graph is contained in $P$ and such that $x E x' \implies f(x) F f(x')$ for all $x, x' \in X$.

	De Rancourt and Miller have shown that these definitions coincide when $F$ is ``strongly idealistic'', which happens in particular when $E$ is countable or is generated by a Borel action of a Polish group (see the second paragraph of \cite{dRM} for a definition of strongly idealistic Borel equivalence relations, and \cite[Proposition~2.8]{dRM} for a proof that in this case these notions of uniformization are equivalent).

	It is unclear what one can say in terms of Borel $E \times F$-invariant uniformizing functions when $F$ is not strongly idealistic. For example, there are smooth Borel equivalence relations $E$ which do not admit a Borel selector (see e.g. \cite[Exercise~5.4.6]{Gao-IDST}), in which case $P = E$ does not admit a Borel $E \times F$-invariant uniformizing function, but does admit a Borel $E \times F$-invariant uniformization as defined above.
\end{rmk}

\paragraph{Acknowledgements.} Research partially supported by NSF Grant DMS-1950475. This paper appears as a chapter in the second author's Ph.D. thesis. We would like to thank Todor Tsankov who asked whether measure invariant uniformization holds for countable Borel equivalence relations. We would also like to thank Ben Miller, Andrew Marks and Dino Rossegger for useful comments and discussion. Finally, we thank the anonymous referee for many helpful comments and suggestions.

\section{Borel invariant uniformization and smoothness}\label{sec:equivalence-smooth}

\subsection{Smooth Borel equivalence relations}\label{subsec:smooth-case}

We begin by proving that smooth Borel equivalence relations satisfy (ii)-(v) of \cref{thm:equivalence-uniformization-smooth}. We first show that (i) implies (ii), the proof that (i) implies (iii) being similar. Fix a Polish space $Z$ and a Borel function $S\colon X\to Z$ such that $x_1 E x_2 \iff S(x_1) = S(x_2)$. Fix also $Y,\mu, P$ as in the definition of measure invariant uniformization. Define $P^*\subseteq Z\times Y$ as follows:

\[
(z,y) \in P^* \iff \forall x\in X \big( S(x) = z \implies (x,y) \in P\big).
\]
Then $P^*$ is $\boldsymbol{\Pi^1_1}$ and we have that

\[
S(x) =z \implies P^*_z = P_x,
\]
\[
z\notin S(X) \implies P^*_z = Y.
\]
Thus $\forall z\in Z (\mu (P^*_z) > 0)$. Then, by measure uniformization for $\bm{\Pi^1_1}$ sets \cite[36.24]{CDST}, there is a Borel function $f^*\colon Z \to Y$ such that $\forall z \in Z ( (z,f^* (z) ) \in P^*)$. Put
\[
f(x) = f^* (S(x)).
\]
Then $f$ is a Borel $E$-invariant uniformization of $P$.

We next prove that (i) implies (iv) (and therefore (v)). For this, we will need the following theorem of Hurewicz on the complexity of sections of a Borel set.

\begin{thm}[Hurewicz, see {\cite[35.47]{CDST}}]\label{thm:Hurewicz}
	Let $X, Y$ be Polish spaces. For any Borel set $B \subseteq X \times Y$,
	\[\{x \in X : B_x ~\text{is}~ K_\sigma\}\]
	is $\bm{\Pi^1_1}$, and the same holds when $K_\sigma$ is replaced with $F_\sigma$ or $G_\delta$.
\end{thm}

We will also need the following generalization of the countable and $K_\sigma$ uniformization theorems given in \cref{thm:K-sigma-uniformization,thm:countable-uniformization}. We give here both the classical statements and their effective refinements, which we will need in later arguments.

\begin{thm}\label{thm:eff-ln-ak-complete}
	Fix Polish spaces $X, Y$ and $P \subseteq X \times Y$.
	\begin{enumerate}
		\item \normalfont{(Lusin--Novikov \cite[18.10]{CDST})} If $P$ is Borel and has countable sections, then $\proj_X(P)$ is Borel. Moreover, there are (partial) Borel functions $f_n: X \to Y$ with $P = \bigcup_n \graph(f_n)$. In particular, $P$ admits a Borel uniformization.

		\item \normalfont{(Effective Lusin--Novikov \cite[4F.6, 4F.17]{Mo})} If $P$ is (lightface) $\Delta^1_1(r)$ and has countable sections, then $\proj_X(P)$ is $\Delta^1_1(r)$. Moreover, there are (partial) $\Delta^1_1(r)$ functions $f_n: X \to Y$ with $P = \bigcup_n \graph(f_n)$. In particular, $P$ admits a $\Delta^1_1(r)$ uniformization.

		\item \normalfont{(Arsenin, Kunugui, Saint Raymond \cite[35.46]{CDST})} If $P$ is Borel and has $K_\sigma$ sections, then $\proj_X(P)$ is Borel. Moreover, $P$ admits a Borel uniformization, and there is a sequence of Borel functions $K_n: X \to K(Y)$ with $P_x = \bigcup_n K_n(x)$ for $x \in X$ (here $K(Y)$ denotes the Polish space of compact subsets of $Y$ \cite[4.F]{CDST}).

		\item \normalfont{(Effective $K_\sigma$ uniformization \cite[4F.16, 4F.18]{Mo})} If $P$ is $\Delta^1_1(r)$ and has $K_\sigma$ sections, then $\proj_X(P)$ is $\Delta^1_1(r)$. Moreover, $P$ admits a $\Delta^1_1(r)$ uniformization, and there is a sequence of $\Delta^1_1(r)$ functions $K_n: X \to K(Y)$ with $P_x = \bigcup_n K_n(x)$ for $x \in X$.
	\end{enumerate}
\end{thm}

Fix $Z, S$ as in the previous case and $Y, P$ as in the definition of $K_\sigma$ invariant uniformization. Define $P^*$ as before. Then $A = \{(z,y)\colon \exists x \in X (S(x) = z \And P(x, y))\}$ is a $\boldsymbol{\Sigma^1_1}$ subset of $P^*$, so by the Lusin Separation Theorem there is a Borel subset $P^{**}$ of $P^*$ such that $A\subseteq P^{**}$. By \cref{thm:Hurewicz}, the set $C$ of all $z\in Z$ such that $P^{**}_z$ is $K_\sigma$ is $\boldsymbol{\Pi^1_1}$ and contains the $\boldsymbol{\Sigma^1_1}$ set $S(X)$, so by the Lusin Separation Theorem there is a Borel set $B$ with $S(X)\subseteq B\subseteq C$. Then if $Q\subseteq Z\times Y$ is defined by
\[
(z,y) \in Q \iff z\in B \ \&  \ (z,y) \in P^{**},
\]
we have that
\[
S(x) =z \implies Q_z = P_x
\]
and every $Q_z$ is $K_\sigma$. It follows, by \cref{thm:eff-ln-ak-complete},  that $D = \proj_Z (Q)$ is Borel and there is a Borel function $g\colon D\to Y$ such that $\forall z\in D (z, g(z) )\in Q$. Since $S(X)\subseteq D$, the function

\[
f(x) = g (S(x))
\]
is a Borel $E$-invariant uniformization of $P$.

\subsection{Non-smooth Borel equivalence relations}\label{subsec:non-smooth-case}

We will now show that non-smooth Borel equivalence relations do not satisfy (ii), (iii) and (v) (and thus also not (iv)) of \cref{thm:equivalence-uniformization-smooth}. We will use the following lemma. Below for Borel equivalence relations $E,E'$ on Polish spaces $X,X'$, resp., we write $E\leq_B E'$ iff there is a Borel map $f\colon X\to X'$ such that $x_1 E x_2 \iff f(x_1) E' f(x_2)$, i.e., $E$ can be \tb{Borel reduced} to $E'$ (via the reduction $f$).

\begin{lem}\label{lem:uniformization-closed-reducibility}
Let $E,E'$ be  Borel equivalence relations on Polish spaces $X,X'$, resp., such that $E\leq_B E'$. If $E$ fails (ii) (resp., (iii), (iv), (v)), so does $E'$.
\end{lem}

\begin{proof}
Let $f\colon X\to X'$ be a Borel reduction of $E$ into $E'$. Assume first that $E$ fails (ii) with witness $Y,
\mu, P$. Define $P'\subseteq X'\times Y$ by
\[
 (x', y)\in P' \iff \forall x\in X\Big(f(x) E' x'
 \implies  (x, y)\in P \Big).
\]
Then note that
\[
f(x) E' x' \implies P'_{x'} = P_x,
\]
\[
x'\notin [f(X)]_{E'} \implies P'_{x'} = Y.
\]
Clearly $P'$ is $\boldsymbol{\Pi^1_1}$ and invariant under the Borel equivalence relation $E'\times \Delta_Y$.

Consider now the $\boldsymbol{\Sigma^1_1}$ subset $P''$ of $P'$ defined by
\[
(x',y)\in P'' \iff \exists x\in X \Big( f(x) E' x' \ \& \ (x,y)\in P\Big).
\]
Note that $P_x = P'_{x'} = P''_{x'}$ whenever $f(x) = x'$. By the Lusin Separation Theorem, fix a Borel set $P'''_0$ satisfying $P'' \subseteq P'''_0 \subseteq P'$ and recursively set $P'''_{n+1}$ to be a Borel set satisfying $[P'''_n]_{E' \times \Delta_Y} \subseteq P'''_{n+1} \subseteq P'$. Then $P''' = \bigcup_n P'''_n$ is a Borel $E' \times \Delta_Y$-invariant set satisfying $P'' \subseteq P''' \subseteq P'$.

Let now $Z\subseteq X'$ be defined by
\[
x'\in Z \iff \mu (P^{'''}_{x'}) >0.
\]
Then $Z$ is Borel and $E'$-invariant and contains $[f(X)]_{E'}$. Finally, define $Q\subseteq X'\times Y$ by
\[
(x', y) \in Q \iff \big( x'\in Z \ \& \ (x', y) \in P'''\big) \  \textrm{or} \ x'\notin Z.
\]
Then $f(x) = x' \implies P''_{x'} \subseteq Q_x \subseteq P'_{x'} \implies Q_{x'} = P_x$, so a Borel $E'$-invariant uniformization of $Q$ would induce a Borel $E$-invariant uniformization of $P$, and hence $Y,\mu, Q$ witnesses the failure of (ii) for $E'$.

The case of (iii) is similar and we next consider the case of (iv). Repeat then the previous argument for case (ii) until the definition of $P'''$. Then define $Z'\subseteq X'$ by
\[
x'\in Z' \iff P^{'''}_{x'} \ \textrm{is} \ K_\sigma \ \textrm{and nonempty}.
\]
We claim that $Z'$ is $\bm{\Pi^1_1}$. This follows from \cref{thm:Hurewicz}, the fact that every nonempty (lightface) $\Delta^1_1$ $K_\sigma$ set contains a $\Delta^1_1$ element, and the theorem on restricted quantification for $\Delta^1_1$:

\begin{thm}[See {\cite[4F.15]{Mo}}]\label{thm:Ksigma-delta11-point}
	Let $K \subseteq \R^\N$ be a nonempty (lightface) $\Delta^1_1(r)$ $K_\sigma$ set, for some $r \in \R^\N$. Then $K \cap \Delta^1_1(r) \neq \emptyset$.
\end{thm}

\begin{thm}[Kleene, see {\cite[4D.3]{Mo}}]\label{thm:restricted-quantification}
	If $Q \subseteq (\R^\N)^3$ is $\Pi^1_1(r)$, for some $r \in \R^\N$, and
	\[R(x, z) \iff \exists y \in \Delta^1_1(r, z) (Q(x, z, y)),\]
	then $R$ is $\Pi^1_1(r)$.
\end{thm}

By \cite[4.17]{CDST} we may assume that $P'''$ is a Borel subset of $\R^\N \times \R^\N$, and hence is $\Delta^1_1(r)$ for some $r \in \R^\N$. By \cref{thm:Ksigma-delta11-point} this gives
\begin{align*}
	x' \in Z' & \iff P^{'''}_{x'} \ \textrm{is} \ K_\sigma \ \textrm{and nonempty} \\
	& \iff P^{'''}_{x'} \ \textrm{is} \ K_\sigma \ \textrm{and} \ \exists y \in \Delta^1_1(r, x') (y \in P'''_{x'}),
\end{align*}
which is $\bm{\Pi^1_1}$ by \cref{thm:Hurewicz,thm:restricted-quantification}.

Now $Z'$ is $\bm{\Pi^1_1}$, $E'$-invariant, and contains $[f(X)]_{E'}$. Let then $Z$ be $E'$-invariant Borel with $[f(X)]_{E'}\subseteq Z\subseteq Z'$, which can be done via the Lusin Separation Theorem as with the construction of $P'''$, and define $Q$ as before but replacing ``$x'\notin Z$'' by ``$(x' \notin Z \ \textrm{and} \ y=y_0)$'', for some fixed $y_0\in Y$. Then $Y, Q$ witnesses the failure of (iv) for $E'$.

Finally, the case of (v) is similar to (iv) by now defining
\[
x'\in Z' \iff P^{'''}_{x'} \ \textrm{is countable and nonempty},
\]
which is $\bm{\Pi^1_1}$ by \cref{thm:Ksigma-delta11-point} and \cite[29.19]{CDST}.
\end{proof}

Assume now that $E$ is not smooth. Then by \cite{HKL} we have $\E_0\leq_B E$. Thus by \cref{lem:uniformization-closed-reducibility} it is enough to show that $\E_0$ fails (ii), (iii), and (v) (thus also (iv)).

We first prove that $\E_0$ fails (ii).  We view here $2^\N$ as the Cantor group $(\Z/2\Z)^\N$ with pointwise addition $+$ and we let $\mu$ be the Haar measure, i.e., the usual product measure. Let then $A\subseteq (\Z/2\Z)^\N$ be an $F_\sigma$ set which has $\mu$-measure 1 but is meager. Let $X=Y = (\Z/2\Z)^\N$ and define $P\subseteq X\times Y$ as follows:
\[
(x,y)\in P \iff \exists x' \E_0 x (x' + y \in A).
\]
Clearly $P$ is $F_\sigma$ and, since $P_x = \bigcup_{x' \E_0 x} (A-x')$, clearly $\mu(P_x) = 1$. Moreover $P$ is $\E_0$-invariant. Assume then, towards a contradiction that $f$ is a Borel $\E_0$-invariant uniformization. Since $x \E_0 x' \implies f(x) = f(x')$, by generic ergodicity of $\E_0$ there is a comeager Borel $\E_0$-invariant set $C\subseteq X$ and $y_0$ such that $\forall x \in C (f(x) = y_0)$; thus $\forall x\in C (x,y_0) \in P$, so $\forall x\in C \exists x' \E_0 x (x' \in A-y_0)$. If $G\subseteq (\Z/2\Z)^\N$ is the subgroup consisting of the eventually 0 sequences, then $x\E_0 y \iff \exists g\in G (g+x = y)$; thus $C=\bigcup_{g\in G} (g + (A-y_0))$, so $C$ is meager, a contradiction.

To show that $\E_0$ fails (v), define
\[
(x,y) \in P \iff x \E_0 y.
\]
Then any Borel $\E_0$-invariant uniformization of $P$ gives a Borel selector for $\E_0$, a contradiction.

Finally to see that $\E_0$ fails (iii), use above $B = (\Z/2\Z)^\N\setminus A$, instead of $A$, to produce a $G_\delta$ set $Q$ as follows:
\[
(x,y)\in Q \iff \forall x' \E_0 x (x' + y \in B).
\]
Then $Q$ is $\E_0$-invariant and has comeager sections. If $g$ is a Borel $\E_0$-invariant uniformization, then by the ergodicity of $\E_0$, there is a $\mu$-measure 1 set $D$ and $y_0$ such that $\forall x\in D \forall x' \E_0 x (x' \in B-y_0)$, so $D\subseteq B-y_0$, and thus $\mu (D) = 0$, a contradiction.

This completes the proof of \cref{thm:equivalence-uniformization-smooth}.

\begin{rmk}
	Andrew Marks and Dino Rossegger have pointed out that the construction in \cref{prop:Ectbl} actually gives a strengthening of \cref{thm:equivalence-uniformization-smooth}: If $E$ is not smooth then $E$ fails co-countable invariant uniformization, i.e., there is an $E$-invariant Borel set $P$ whose sections are co-countable which does not admit a Borel $E$-invariant uniformization.

	To see this, define a hyperfinite Borel equivalence relation $E$ on $(2^\N)^\N$ by $x E y$ iff there is a permutation $\sigma$ of $\N$, fixing all but finitely many numbers, so that $x_n = y_{\sigma(n)}$ for $n \in \N$. H. Friedman has shown the following strengthening of \cref{thm:friedman} for this equivalence relation \cite[Proposition~C]{Friedman}: If $F: (2^\N)^\N \to 2^\N$ is Borel and $E$-invariant, then there is some $x \in (2^\N)^\N$ such that $F(x) = x_0$.

	Let now $P$ be as in the proof of \cref{prop:Ectbl}. Then $P$ is Borel, $E$-invariant, and has co-countable sections. If $F$ were a Borel $E$-invariant uniformization of $P$, then there would be some $x$ with $F(x) = x_0$. But by definition $\lnot P(x, x_0)$, a contradiction.

	Let now $E'$ be a non-smooth Borel equivalence relation on $X$. By \cite{HKL,DJK}, if $E'$ is not smooth then there is a Borel reduction $f$ from $E$ to $E'$. Define $P'$ as in the proof of \cref{lem:uniformization-closed-reducibility}. Then $P'$ has co-countable sections and does not admit a Borel $E'$-invariant uniformization, so it remains to check that $P'$ is Borel. To see this, write
	\[Q(x', y, x) \iff f(x) E' x' \And \lnot P(x, y).\]
	Then $Q$ is Borel and its sections $Q_{(x', y)}$ are either an $E$-class or empty, hence countable. Thus by the Lusin--Novikov Theorem $P' = X \times 2^\N \setminus \proj_{X \times 2^\N}(Q)$ is Borel.
\end{rmk}

\subsection{A strengthening of the smooth case}\label{subsec:smooth-strengthening}

We note the following strengthening of \cref{thm:equivalence-uniformization-smooth} in the case that $E$ is smooth, where $K(Y)$ denotes the Polish space of compact subsets of $Y$ \cite[4.F]{CDST}:

\begin{thm}\label{thm:smooth-lusin-novikov}
	Let $X, Y$ be Polish spaces, $E$ be a smooth Borel equivalence relation on $X$, and $P \subseteq X \times Y$ be a Borel $E$-invariant set with non-empty sections.
	\begin{enumerate}
		\item If $P$ has countable sections, then $P = \bigcup_n \graph(g_n)$ for a sequence of $E$-invariant Borel maps $g_n: X \to Y$.
		\item If $P$ has $K_\sigma$ sections, then $P_x = \bigcup_n K_n(x)$ for a sequence of $E$-invariant Borel maps $K_n: X \to K(Y)$.
		\item If $P$ has comeager sections, then $P \supseteq \bigcap_n U_n$ for a sequence of $E$-invariant Borel sets $U_n \subseteq X \times Y$ with dense open sections. Moreover, if $P$ has dense $G_\delta$ sections, we can find such $U_n$ with $P = \bigcap_n U_n$.
	\end{enumerate}
\end{thm}

\begin{proof}
	Let $Z, S, P^*, P^{**}$ be as in the proof of (i) $\implies$ (iv) of \cref{thm:equivalence-uniformization-smooth}.

	The second assertion follows from \cref{thm:eff-ln-ak-complete} applied to $Q$ from the proof of (i) $\implies$ (iv). Indeed, if $P$ has $K_\sigma$ sections then so does $Q$, so by \cref{thm:eff-ln-ak-complete} there are Borel functions $K'_n$ such that $Q_{x'} = \bigcup_n K'_n(x')$ for all $x'$. The maps $K_n = K'_n \circ S$ witness 2.

	For the first, we construct $Q$ as in the proof of (i) $\implies$ (iv), except we let $C$ be the set of all $z \in Z$ such that $P^{**}_z$ is countable, which is $\bm{\Pi^1_1}$ by \cite[29.19]{CDST}. Thus $Q$ has countable sections, so 1 follows from the Lusin--Novikov Theorem applied to $Q$.

	For the third, let $Z, S, P^*, P^{**}$ be as in the proof of (i) $\implies$ (iv). By \cite[16.1]{CDST} the set $C$ of all $z$ for which $P^{**}_z$ is comeager is Borel, so $Q(z, y) \iff [C(z) \implies P^{**}(z, y)]$ is Borel with comeager sections and $S(x) = z \implies P_x = Q_z$.

	If moreover $P$ has $G_\delta$ sections, we instead let $A$ be the set of all $z$ for which $P^{**}_z$ is comeager and $G_\delta$, which is $\bm{\Pi^1_1}$ by \cref{thm:Hurewicz}. Then $S(X) \subseteq A$ is $\bm{\Sigma^1_1}$, so by the Lusin Separation Theorem there is a Borel set $S(X) \subseteq C \subseteq A$. We then define $Q$ as above, so that $Q$ moreover has $G_\delta$ sections.

	By \cite[35.36, 35.43]{CDST}, there are Borel sets $U_n' \subseteq Z \times Y$ with dense open sections satisfying $Q \supseteq \bigcap_n U_n'$ (or, in the case that $Q$ has $G_\delta$ sections, $Q = \bigcap_n U_n'$). We then take $(x, y) \in U_n \iff (S(x), y) \in U_n'$.
\end{proof}

\subsection{Effective consequences and considerations}\label{subsec:equiv-effective-consequences}

\Cref{thm:measure-uniformization,thm:category-uniformization,thm:K-sigma-uniformization} are effective, meaning that whenever $P$ is (lightface) $\Delta^1_1$ and satisfies the hypotheses of one of these theorems, then $P$ admits a $\Delta^1_1$ uniformization (see \cref{thm:eff-ln-ak-complete}, as well as \cite[4F.20]{Mo} and the subsequent discussion). Similarly, in \cite{HKL} it is shown that if $E$ is smooth and $\Delta^1_1$ then it has a $\Delta^1_1$ reduction to $\Delta(2^\N)$. The proof of \cref{thm:equivalence-uniformization-smooth} therefore gives the following effective refinement:

\begin{thm}\label{thm:effective-equivalence-uniformization-smooth}
	Let $E$ be a smooth $\Delta^1_1$ equivalence relation on $\N^\N$ and $P \subseteq \N^\N \times \N^\N$ be $\Delta^1_1$ and $E$-invariant. Then $P$ admits a $\Delta^1_1$ $E$-invariant uniformization whenever one of the following holds:
	\begin{enumerate}[label=(\roman*)]
		\item $P$ has $\mu$-positive sections, for some $\Delta^1_1$ probability measure $\mu$ on $\N^\N$;
		\item $P$ has non-meager sections;
		\item $P$ has non-empty $K_\sigma$ sections;
		\item $P$ has non-empty countable sections.
	\end{enumerate}
\end{thm}

In (i) above, we identify probability Borel measures on $\N^\N$ with points in $[0, 1]^{\N^{<\N}}$; see \cite[17.7]{CDST}.

It is also interesting to consider whether the converse holds. For example, let $E$ be a $\Delta^1_1$ equivalence relation on $\N^\N$, and suppose that for every $\Delta^1_1$ $E$-invariant set $P \subseteq \N^\N \times \N^\N$ which satisfies one of (i)-(iv) above, $P$ admits a $\Delta^1_1$ $E$-invariant uniformization. Must it be the case that $E$ is smooth?

If we replace $\Delta^1_1$ by Borel, then $E$ must indeed be smooth by \cref{thm:equivalence-uniformization-smooth}. However, to prove this we use the fact that every non-smooth $\Delta^1_1$ equivalence relation embeds $\E_0$ \cite{HKL}, and this is not effective: There are non-smooth $\Delta^1_1$ equivalence relations on $\N^\N$ which do not admit $\Delta^1_1$ embeddings of $\E_0$.

\begin{eg}
	Let $A \subseteq \N^\N$ be a nonempty $\Pi^0_1$ set containing no $\Delta^1_1$ points (c.f. \cite[4D.10]{Mo}). Define an equivalence relation $E$ on $\N^\N \times 2^\N \cong \N^\N$ by
	\[(x, y) E (x', y') \iff (x, x' \notin A) \lor (x, x' \in A \And y \E_0 y'),\]
	and note that $E$ is $\Delta^1_1$ and non-smooth because for any $x \in A$ we have that $y \mapsto (x, y)$ is a Borel embedding of $\E_0$ into $E$. If $f: 2^\N \to \N^\N \times 2^\N$ were a $\Delta^1_1$ reduction of $\E_0$ into $E$, then there would be at most one $\E_0$-class $C$ such that $\proj_0(f(x)) \notin A$ for $x \in C$. But then for any $x \in \Delta^1_1 \setminus C$ we would have $\proj_0(f(x)) \in \Delta^1_1 \cap A$, a contradiction.
\end{eg}

Restricting our attention to those $P$ which have countable sections, it turns out that the converse to \cref{thm:effective-equivalence-uniformization-smooth} is false. In fact, using the theory of turbulence, one can construct the following very strong counterexample:

\begin{thm}\label{thm:turbulence-example}
	There is a $\Pi^0_1$ set $N \subseteq \N^\N$ and a $\Delta^1_1$ equivalence relation $E$ on $N$ which is not smooth, and such that every $\Delta^1_1$ $E$-invariant set $P \subseteq N \times \N^\N$ with non-empty countable sections is invariant, meaning that $P_x = P_{x'}$ for all $x, x' \in N$.
\end{thm}

\begin{cor}\label{cor:effective-non-smooth}
	There is a $\Delta^1_1$ equivalence relation $F$ on $\N^\N$ which is not smooth, and such that every $\Delta^1_1$ $F$-invariant set $P \subseteq \N^\N \times \N^\N$ with non-empty countable sections admits a $\Delta^1_1$ $F$-invariant uniformization.
\end{cor}

\begin{proof}
	Let $N, E$ be as in \cref{thm:turbulence-example} and define
	\[x F x' \iff (x = x') \lor (x, x' \in N \And x E x').\]

	Suppose now that $P \subseteq \N^\N \times \N^\N$ were $\Delta^1_1$, $F$-invariant, and had non-empty countable sections. Let $y \in A \iff \exists x \in N (P(x, y)) \iff \forall x \in N(P(x, y))$. Then $A$ is countable and $\Delta^1_1$, and $P_x = A$ for all $x \in N$. In particular, by \cref{thm:Ksigma-delta11-point} $A$ contains a $\Delta^1_1$ point, say $y_0$.

	By the effective Lusin--Novikov Theorem there is a $\Delta^1_1$ uniformization $f$ of $P$. Letting $g(x) = f(x)$ for $x \notin N$, and $g(x) = y_0$ otherwise, gives a $\Delta^1_1$ $F$-invariant uniformization of $P$.
\end{proof}

\begin{rmk}
	The equivalence relation $F$ constructed in \cref{cor:effective-non-smooth} gives another example of a non-smooth $\Delta^1_1$ equivalence relation which does not admit a $\Delta^1_1$ reduction of $\E_0$. Indeed, if $f$ were a $\Delta^1_1$ reduction of $\E_0$ to $F$ and $y_0 \in \Delta^1_1$ then
	\[P = \{(x, y) \in \N^\N \times 2^\N : \exists x' \in 2^\N (f(x') F x \And y \E_0 x') \lor (x \notin [f(2^\N)]_F \And y = y_0)\}\]
	would be a $\Delta^1_1$ $F$-invariant set with non-empty countable sections, by the effective Lusin--Novikov Theorem. But then $P$ admits a $\Delta^1_1$ $F$-invariant uniformization $g$, and $g \circ f$ is be a $\Delta^1_1$ selector for $\E_0$, a contradiction.
\end{rmk}

Fix now a parametrization of the $\Delta^1_1$ sets in $\N^\N$ (see \cite[Section~3H]{Mo}). That is, fix a set $\bbD \subseteq \N$ and two sets $S, P \in \N \times \N^\N$ such that
\begin{enumerate}[label=(\roman*)]
	\item $\bbD$ is $\Pi^1_1$, $S$ is $\Sigma^1_1$ and $P$ is $\Pi^1_1$;
	\item for $d \in \bbD$, $S_d = P_d$, and we denote this set by $\bbD_d$; and
	\item every $\Delta^1_1$ set in $\N^\N$ appears as $\bbD_d$ for some $d \in \bbD$.
\end{enumerate}
We may similarly fix a parametrization of the $\Delta^1_1$ functions from $\R^\N$ to $\N^\N$, namely, sets $\bbF \subseteq \N$, $S', P' \subseteq \N \times \R^\N \times \N^\N$ such that
\begin{enumerate}[label=(\roman*)]
	\item $\bbF$ is $\Pi^1_1$, $S'$ is $\Sigma^1_1$ and $P'$ is $\Pi^1_1$;
	\item for $n \in \bbF$, $S'_n = P'_n$, and we denote this set by $\bbF_n$; and
	\item every $\Delta^1_1$ function from $\R^\N$ to $\N^\N$ appears as $\bbF_n$ for some $n \in \bbF$.
\end{enumerate}
To see this, let $\bbD'$ be a parametrization of the $\Delta^1_1$ sets in $\R^\N \times \N^\N$, and note that
\begin{align*}
	n \in \bbF \iff & n \in \bbD' \And \bbD'_n \text{ is a function} \\
	\iff & n \in \bbD' \And \forall x \exists y \in \Delta^1_1(x) P_n(x, y) \And \\
	& \forall x, y, y' (S_n(x, y) \And S_n(x, y') \implies y = y')
\end{align*}
is $\Pi^1_1$ by \cref{thm:Ksigma-delta11-point,thm:restricted-quantification}. (These are the analogues of the coding of \cref{sec:hardness-proofs} for the $\Delta^1_1$ sets and functions.)

\begin{proof}[Proof of \cref{thm:turbulence-example}]
	Consider the group $\R^\N$ and the translation action of $\ell^1 \subseteq \R^\N$ on $\R^\N$, which is turbulent by \cite[Section~10(ii)]{kechris_turbulence}. Let $F$ be the induced equivalence relation, which is clearly $\Delta^1_1$.

	Let $\bbD, \bbF$ be the parametrizations of the $\Delta^1_1$ sets and functions we fixed above. Let $\bbH \subseteq \bbF$ be the set of those $n \in \bbF$ for which $\bbF_n$ is a homomorphism of $F$ into $E^{\N^\N}_{ctble}$, and note that $\bbH$ is $\Pi^1_1$ as
	\[n \in \bbH \iff n \in \bbF \And \forall x, x' (x F x' \implies \bbF_n(x) E^{\N^\N}_{ctble} \bbF_n(x')).\]

	By the proof of \cite[Theorem~12.5(i)$\implies$(ii)]{kechris_turbulence}, for each $n \in \bbH$ there is $\Delta^1_1$ comeager $F$-invariant set $C_n \subseteq \R^\N$ which $\bbF_n$ maps to a single $\E^{\N^\N}_{ctble}$-class. Moreover, there is a computable map $n \mapsto n^*$ such that if $n \in \bbH$ then $n^* \in \bbD$ and $C_n = \bbD_{n^*}$.

	Put $C = \bigcap_{n \in \bbH} C_n \subseteq \R^\N$. Then $C$ is comeager, $F$-invariant and $\Sigma^1_1$, since
	\[a \in C \iff \forall n (n \in \bbH \implies a \in \bbD_{n^*}).\]
	Moreover, for each $\Delta^1_1$ homomorphism $f$ of $F$ to $\E^{\N^\N}_{ctble}$, $\res{f}{C}$ maps into a single $\E^{\N^\N}_{ctble}$-class.

	Let now $N \subseteq \N^\N$ be $\Pi^0_1$ and $c: N \to \R^\N$ be a $\Delta^1_1$ map such that $c(N) = C$. Define the $\Delta^1_1$ equivalence relation $E$ on $N$ by
	\[x E x' \iff c(x) F c(x').\]
	We will show that this $E$ works.

	Let $P \subseteq N \times \N^\N$ be $\Delta^1_1$ and $E$-invariant with non-empty countable sections. Define $Q \subseteq C \times \N^\N$ by
	\begin{align*}
		(a, y) \in Q &\iff a \in C \And \exists x \in N (c(x) = a \And P(x, y)) \\
		&\iff a \in C \And \forall x \in N (c(x) = a \implies P(x, y)).
	\end{align*}
	Note that $Q$ is $F$-invariant. Moreover, $Q$ is $\Delta^1_1$ on the $\Sigma^1_1$ set $C \times \N^\N$, i.e., it is the intersection of $C \times \N^\N$ with a $\Sigma^1_1$ set in $\R^\N \times \N^\N$ as well as with a $\Pi^1_1$ set in $\R^\N \times \N^\N$. By $\Sigma^1_1$ separation, there is a $\Delta^1_1$ set $R \subseteq \R^\N \times \N^\N$ such that $R \cap (C \times \N^\N) = Q$.

	Let now $C^* \subseteq \R^\N$ be defined by
	\[a \in C^* \iff \forall a' [a F a' \implies R_a = R_{a'} \And R_a \text{ is countable and non-empty}].\]
	We claim that $C^*$ is $\Pi^1_1$. Indeed, by the Effective Perfect Set Theorem \cite[4F.1]{Mo}, $R_a$ is countable iff $R_a \subseteq \Delta^1_1(a)$, for all $a \in \R^\N$. Thus
	\[R_a \text{ is countable} \iff \forall x (x \in R_a \implies x \in \Delta^1_1(a)),\]
	which is $\Pi^1_1$ by \cref{thm:restricted-quantification} (see also \cite[4D.14]{Mo}). Moreover, when $R_a$ is countable, checking that it is non-empty is $\Pi^1_1$ by \cref{thm:restricted-quantification}.

	Thus $C^*$ is $\Pi^1_1$, $F$-invariant and contains $C$. By the effective Lusin Separation Theorem \cite[4B.11]{Mo}, there is an increasing sequence of $\Delta^1_1$ sets $B_n$ such that $C \subseteq B_n \subseteq C^*$ and $[B_n]_F \subseteq B_{n+1}$, and such that the union $B = \bigcup_n B_n$ is $\Delta^1_1$. There is therefore a $\Delta^1_1$ set $B$ which is $F$-invariant and such that $C \subseteq B \subseteq C^*$. Finally, define $S \subseteq \R^\N \times \N^\N$ by
	\[(a, y) \in S \iff [a \in B \And R(a, y)] \lor [a \notin B \And y = y_0]\]
	for some fixed $\Delta^1_1$ point $y_0$ in $\N^\N$. Then $S$ is $\Delta^1_1$, $F$-invariant, and has non-empty countable sections.

	Let $s: \R^\N \to (\N^\N)^\N$ be a $\Delta^1_1$ homomorphism of $F$ to $\E^{\N^\N}_{ctble}$ for which $S_a = \{s(a)_n\}$ for all $a \in \R^\N$, which exists by the effective Lusin--Novikov Theorem. By the definition of $C$, we have that $\res{s}{C}$ maps into a single $\E^{\N^\N}_{ctble}$-class. Let $A$ be the corresponding countable set. Then for $a \in C$ and any $x \in N$ with $c(x) = a$,
	\[A = S_a = R_a = Q_a = P_x,\]
	so $P_x = A$ for all $x \in N$.

	It remains to check that $E$ is not smooth. To see this, note that $\res{F}{C}$ has at least two classes (as every $F$-class is meager), and hence so does $E$. If $E$ were smooth, then there would be a $\Delta^1_1$ map $f: N \to \N^\N$ for which
	\[x E x' \iff f(x) = f(x').\]
	But then $P = \graph(f)$ would be $\Delta^1_1$, $E$-invariant, have non-empty countable sections, and satisfy $P_x \neq P_{x'}$ for some $x, x' \in N$, a contradiction.
\end{proof}

When $E$ is $\Delta^1_1$ and smooth, we have seen in \cref{thm:effective-equivalence-uniformization-smooth} that the effective analogue of \cref{thm:equivalence-uniformization-smooth} holds for $E$. On the other hand, by \cref{cor:effective-non-smooth} there is a non-smooth $\Delta^1_1$ equivalence relation $F$ for which every $\Delta^1_1$ $F$-invariant set with countable sections admits a $\Delta^1_1$ $F$-invariant uniformization, i.e., \cref{thm:effective-equivalence-uniformization-smooth} (iv) holds for a non-smooth $\Delta^1_1$ equivalence relation $F$. It is an interesting question whether this can happen when $P$ has ``large'' or $K_\sigma$ sections.

\begin{prob}
	Is there a $\Delta^1_1$ equivalence relation $E$ on $\N^\N$ which is not smooth, and such that all $\Delta^1_1$ $E$-invariant sets $P \subseteq \N^\N \times \N^\N$ satisfying one of (i)-(iii) in \cref{thm:effective-equivalence-uniformization-smooth} admit a $\Delta^1_1$ $E$-invariant uniformization?
\end{prob}

Finally, we remark that if $E$ is a $\Delta^1_1$ equivalence relation which is not smooth then there is a continuous embedding of $\E_0$ into $E$ which is recursive in Kleene's $\mathcal{O}$, the universal $\Pi^1_1$ subset of $\N$. In particular, the converse of \cref{thm:effective-equivalence-uniformization-smooth} holds if we consider all such $P \in \Delta^1_1(\mathcal{O})$.

\subsection{Proof of \texorpdfstring{\cref{thm:equiv-unif-smooth-jvn-nk}}{Theorem~\ref*{thm:equiv-unif-smooth-jvn-nk}}}\label{subsec:proof-equiv-other-unif}

We now establish \cref{thm:equiv-unif-smooth-jvn-nk}. Let $X, Y$ be Polish spaces, $E$ be a Borel equivalence relation on $X$ and $P \subseteq X \times Y$ be $E$-invariant with non-empty sections.

Let $E$ be smooth, in order to prove (i) $\implies$ (ii), (iii). Suppose first that $P$ is $\bm{\Pi^1_1}$ and define $Z, S, P^*$ as in \cref{subsec:smooth-case} above. Then $P^*$ is $\bm{\Pi^1_1}$ and has non-empty sections, so by Novikov--Kond\^o uniformization it admits a uniformizing function $f^*$ whose graph is $\bm{\Pi^1_1}$. Put $f(x) = f^*(S(x))$. Then $f$ is an $E$-invariant uniformization of $P$, and it's graph is $\bm{\Pi^1_1}$ as
\[f(x) = y \iff \forall z \in Z (S(x) = z \implies f^*(z) = y).\]
Suppose now that $P$ is $\bm{\Sigma^1_1}$ and let $Z, S$ be as in \cref{subsec:smooth-case}. Apply Jankov, von Nenumann uniformization to the set
\[\{(z, y) : \exists x \in X(S(x) = z \And P(x, y))\}\]
to get the (partial) $\sigma(\bm{\Sigma^1_1})$-measurable uniformizing function $f^*$. Then $f(x) = f^*(S(x))$ is $\sigma(\bm{\Sigma^1_1})$-measurable, $E$-invariant and uniformizes $P$.

To show that Jankov--von Neumann and Novikov--Kond\^o invariant uniformization fail for non-smooth Borel equivalence relations, we prove the analogue of \cref{lem:uniformization-closed-reducibility}:

\begin{lem}
	Let $E, E'$ be Borel equivalence relations on Polish spaces $X, X'$ resp. such that $E \leq_B E'$. If $E$ fails Novikov--Kond\^o invariant uniformization, so does $E'$. If $E$ is countable and fails Jankov--von Neumann invariant uniformization, so does $E'$.
\end{lem}

\begin{proof}
	Let $f: X \to X'$ be a Borel reduction of $E$ into $E'$. Assume that $E$ fails Novikov--Kond\^o invariant uniformization with witness $Y, P$. Define $P'$ as in the proof of \cref{lem:uniformization-closed-reducibility}. If $g$ is an $E'$-invariant uniformization of $P'$ with $\bm{\Pi^1_1}$ graph, then $g \circ f$ is an $E$-invariant uniformization of $P$ with $\bm{\Pi^1_1}$ graph. Thus $P'$ witnesses the failure of Novikov--Kond\^o invariant uniformization for $E'$.

	Suppose now that $E$ is countable and fails Jankov--von Neumann invariant uniformization with witness $Y, P$. Define $P'' \subseteq X' \times Y$ as in the proof of \cref{lem:uniformization-closed-reducibility}, so that $P''$ is $\bm{\Sigma^1_1}$, $P''_{x'} = P_x$ whenever $f(x) E' x'$, and $P''_{x'} = \emptyset$ whenever $x' \notin [f(X)]_{E'}$. As $E$ is countable and $f$ is a reduction, if $x' \in [f(X)]_{E'}$ then there are only countably-many $x \in X$ with $f(x) E' x'$, so $[f(X)]_{E'}$ is Borel by the Lusin--Novikov Theorem. Thus
	\[(x', y) \in P''' \iff (x', y) \in P'' \lor x' \notin [f(X)]_{E'}\]
	is $\bm{\Sigma^1_1}$, $E'$-invariant, has non-empty sections, and witnesses the failure of Jankov--von Neumann invariant uniformization for $E'$: if $g$ were a $\sigma(\bm{\Sigma^1_1})$-measurable invariant uniformization of $P'''$, then $g \circ f$ would be a $\sigma(\bm{\Sigma^1_1})$-measurable $E$-invariant uniformization of $P$.
\end{proof}

Thus, as in \cref{subsec:non-smooth-case}, it suffices to show that $\E_0$ fails Jankov--von Neumann and Novikov--Kond\^o invariant uniformization. To see this, define
\[(x, y) \in P \iff x \E_0 y.\]
Then any $\sigma(\bm{\Sigma^1_1})$-measurable $\E_0$-invariant uniformization of $P$ gives a Baire-measurable selector for $\E_0$, a contradiction. Thus $\E_0$ fails Jankov--von Neumann invariant uniformization. If $f$ were an $\E_0$-invariant uniformization of $P$ with $\bm{\Pi^1_1}$ graph, then $A = \{x : f(x) = x\}$ would be a $\bm{\Pi^1_1}$, hence Baire-measurable, transversal of $\E_0$, a contradiction. Thus $\E_0$ fails Novikov--Kond\^o invariant uniformization.

\section{The complexity of counterexamples}\label{sec:complexity-counterexamples}

\subsection{Proof of \texorpdfstring{\cref{thm:complexity-of-counterexamples}}{Theorem~\ref*{thm:complexity-of-counterexamples}}}\label{subsec:minimal-complexity}

We first prove \cref{thm:complexity-of-counterexamples}. Below, we let $F(Y)$ denote the Effros Borel space of closed subsets of $Y$ (cf. \cite[12.C]{CDST}).

\begin{lem}\label{lem:complexity-of-counterexamples-reduction}
	Suppose $X, Y$ are Polish spaces, $E$ is a Borel equivalence relation on $X$, and $P \subseteq X \times Y$ is Borel, $E$-invariant, and has $F_\sigma$ sections. If there is an $E$-invariant Borel map $x \mapsto F_x \in F(Y) \setminus \{\emptyset\}$ such that $P_x$ is non-meager in $F_x$ for all $x \in X$, then $P$ admits a Borel $E$-invariant uniformization.
\end{lem}

\begin{proof}
	By \cite[12.13]{CDST}, there is a sequence of $E$-invariant Borel functions $y_n: X \to Y$ such that $\{y_n(x)\}$ is dense in $F_x$ for all $x \in X$. Since $P_x$ is non-meager and $F_\sigma$ in $F_x$, $P_x$ contains an open set in $F_x$, and in particular contains some $y_n(x)$. Thus the map taking $x \in X$ to the least $y_n(x)$ such that $P(x, y_n(x))$ is an $E$-invariant Borel uniformization of $P$.
\end{proof}

Suppose now that $P$ satisfies one of (i)-(iii) in \cref{thm:complexity-of-counterexamples}. Note that $P$ has $F_\sigma$ sections in all of these cases, so by \cref{lem:complexity-of-counterexamples-reduction} it suffices to construct a Borel map $x \mapsto F_x \in F(Y) \setminus \{\emptyset\}$ such that $P_x$ is non-meager in $F_x$ for all $x \in X$.

In case (ii) we simply take $F_x = Y$.

Consider next case (i), that there is a Borel assignment $x \mapsto \mu_x$ of probability Borel measures on $Y$ such that $\mu_x(P_x) > 0$ and $P_x \in \bm{\Delta^0_2}$ for all $x \in X$. Let $\nu_x$ denote the probability Borel measure $\mu_x$ restricted to $P_x$, i.e., $\nu_x(A) = \mu_x(A \cap P_x) / \mu_x(P_x)$, and define $F_x$ to be the support of $\nu_x$, that is, $F_x$ is the smallest $\nu_x$-conull closed set in $Y$. The map $x \mapsto F_x$ is Borel: if $U \subseteq Y$ is open, then
\[F_x \cap U \neq \emptyset \iff \nu_x(U) > 0 \iff \mu_x(U \cap P_x) > 0,\]
and the set of all $x \in X$ satisfying this is Borel by \cite[17.25]{CDST}.

Because $F_x$ is the support of $\nu_x$, any open set in $F_x$ is $\nu_x$-positive. Any non-meagre $F_\sigma$ set in $F_x$ must contain an open set, and hence must have $\nu_x$-positive measure. Therefore any $\nu_x$-null $F_\sigma$ set in $F_x$ is meagre. Now for all $x \in X$, $P_x$ is $G_\delta$ and $\nu_x$-conull in $F_x$ so $P_x$ is comeager in $F_x$.

Finally consider case (iii), that $P_x \in G_\delta$ and $P_x$ is non-empty and $K_\sigma$ for all $x \in X$. Let $F_x$ be the closure of $P_x$ in $Y$. Then $P_x$ is dense $G_\delta$, hence comeager, in $F_x$, so it remains to check that $x \mapsto F_x$ is Borel. Indeed,
\[F_x \cap U \neq \emptyset \iff P_x \cap U \neq \emptyset,\]
and this is Borel by \cref{thm:eff-ln-ak-complete}, as $P_x \cap U$ is $K_\sigma$ for all $x \in X$.

\subsection{Proof of \texorpdfstring{\cref{thm:ramsey-example}}{Theorem~\ref*{thm:ramsey-example}}}\label{subsec:ramsey-example}

We now prove \cref{thm:ramsey-example}.

Let $X = [\N]^{\aleph_0}$ denote the space of infinite subsets of $\N$. By identifying subsets of $\N$ with their characteristic functions, we can view $X$ as an $\E_0$-invariant $G_\delta$ subspace of $2^\N$. Note that this is a dense $G_\delta$ in $2^\N$, and it is $\mu$-conull, where $\mu$ is the uniform product measure on $2^\N$. We let $E$ denote the equivalence relation $\E_0$ restricted to $X$.

Let $Y = 2^\N$, and define $P \subseteq X \times Y$ by
\[P(A, B) \iff |A \setminus B| = |A \cap B| = \aleph_0.\]
Then $P$ is $G_\delta$ and $E$-invariant, and $P_x$ is comeager for all $x \in X$. By the Borel--Cantelli lemma, one easily sees that $\mu(P_x) = 1$ for all $x \in X$.

We claim that $P$ does not admit an $E$-invariant Borel uniformization. Indeed, suppose such a uniformization $f: X \to Y$ existed. By \cite[19.19]{CDST}, there is some $A \in X$ such that $\res{f}{[A]^{\aleph_0}}$ is continuous, where $[A]^{\aleph_0}$ denotes the space of infinite subsets of $A$. Since $E$-classes are dense in $[A]^{\aleph_0}$, $\res{f}{[A]^{\aleph_0}}$ is constant, say with value $B$. Then $f(A) = B$, so $P(A, B)$ and $A \cap B$ is infinite. But then $A \cap B \in [A]^{\aleph_0}$, so $f(A \cap B) = B$. But $(A \cap B) \setminus B$ is not infinite, so $\lnot P(A \cap B, B)$, a contradiction.

\begin{rmk}
	Using the same Ramsey-theoretic arguments, one can show that the following examples also do not admit $E$-invariant uniformizations:
	\begin{enumerate}
		\item Let $Y$ be the space of graphs on $\N$ and set $Q(A, G)$ iff for all finite disjoint sets $x, y \subseteq \N$ there is some $a \in A$ which is adjacent (in $G$) to every element of $x$ and no element of $y$, i.e., $A$ contains witnesses that $G$ is the random graph.
		\item Let $Y = [\N]^{\aleph_0}$, and for $B \in Y$ let $f_B: \N \to \N$ denote its increasing enumeration. Then take $R(A, B)$ iff $f_B(A)$ contains infinitely many even and infinitely many odd elements.
	\end{enumerate}
	As with $P$ above, $Q, R$ both have $\mu$-conull dense $G_\delta$ sections.
\end{rmk}

\section{Dichotomies and anti-dichotomies}

\subsection{Proof of \texorpdfstring{\cref{thm:ben-dichotomy}}{Theorem~\ref*{thm:ben-dichotomy}}}\label{sec:proof-of-ben-1}

Here we derive Miller's dichotomy \cref{thm:ben-dichotomy} for sets with countable sections, from Miller's $(\G_0, \bbH_0)$ dichotomy \cite{Mi-survey} and Lecomte's $\aleph_0$-dimensional hypergraph dichotomy \cite{L}.

For $s \in 2^{< \N}$, define a Borel graph on $2^\N$ by
\[\G_s = \{(s^\frown (i)^\frown x, s^\frown (1-i)^\frown x) : i < 2, x \in 2^\N\},\]
where ${}^\frown$ denotes concatenation of sequences, and write $\G_S = \bigcup_{s \in S} \G_s$ for $S \subseteq 2^{< \N}$. For $t = (t_0, t_1) \in \bigcup_{n \in \N} 2^n \times 2^n$, define a Borel graph on $2^\N$ by
\[\bbH_t = \{(t_i {}^\frown (i)^\frown x, t_{1-i} {}^\frown (1-i)^\frown x) : i < 2, x \in 2^\N\},\]
and write $\bbH_T = \bigcup_{t \in T} \bbH_t$ for $T \subseteq \bigcup_{n \in \N} 2^n \times 2^n$.

Given $S \subseteq 2^{< \N}$ and $T \subseteq \bigcup_{n \in \N} 2^n \times 2^n$ we say $(S, T)$ is \tb{sparse} if
\[|S \cap 2^n| + |T \cap 2^n \times 2^n| \leq 1\]
for all $n \in \N$. We say $(S, T)$ is \tb{dense} if both $S$ and $T$ are dense, i.e.,
\[\forall u \in 2^{< \N} \exists s \in S (u \subseteq s) ~~\text{and}~~ \forall u \in 2^{< \N} \times 2^{< \N} \exists t \in T \forall i < 2 (u_i \subseteq t_i).\]
Fix a pair $(S, T)$ that is sparse and dense and let $\G_0 = \G_S$ and $\bbH_0 = \bbH_T$.

If $R_i, S_i$ are $d_i$-ary relations on $X, Y$, for $i \in I$, then a \tb{homomorphism} from $(R_i)_{i \in I}$ to $(S_i)_{i \in I}$ is a map $f: X \to Y$ such that $(f(x_n))_{n < d_i} \in S_i$ whenever $(x_n)_{n < d_i} \in R_i$. If $G$ is a hypergraph on a Polish space $X$ we say $G$ has \tb{countable Borel chromatic number}, and write $\chi_B(G) \leq \aleph_0$, if $X$ can be covered by countably-many $G$-independent Borel sets.

\begin{thm}[The $(\G_0, \bbH_0)$ dichotomy, Miller {\cite[Theorem~25]{Mi-survey}}]\label{thm:G0H0}
	Let $X$ be a Polish space, $G$ be an analytic graph on $X$, and $E$ be an analytic equivalence relation on $X$. Then exactly one of the following holds:
	\begin{enumerate}
		\item There is a smooth Borel equivalence relation $F \supseteq E$ such that $\chi_B(G \cap F) \leq \aleph_0$.
		\item There is a continuous homomorphism $\varphi: 2^\N \to X$ from $(\G_0, \bbH_0)$ to $(G, E)$.
	\end{enumerate}
\end{thm}

We will also need the following consequence of the $\aleph_0$-dimensional $\G_0$ dichotomy, where $\sigma(\bm{\Sigma^1_1})$ denotes the $\sigma$-algebra generated by the analytic sets.

\begin{thm}\label{thm:upgrade-colouring}
	Let $X$ be a Polish space and $G$ be an analytic $\aleph_0$-dimensional hypergraph on $X$. If $G$ has a $\sigma(\bm{\Sigma^1_1})$-measurable countable colouring, then $\chi_B(G) \leq \aleph_0$.
\end{thm}

\begin{proof}
	Let $f$ be a $\sigma(\bm{\Sigma^1_1})$-measurable countable colouring of $G$ and suppose for the sake of contradiction that $G$ does not admit a Borel countable colouring. By the $\aleph_0$-dimensional $\G_0$ dichotomy \cite[Theorem~1.6]{L} and \cite[Lemma~2.1]{L} there is a graph $\G_0^\omega$ on $\N^\N$ and a dense $G_\delta$ set $C \subseteq \N^\N$ such that $\res{\G_0^\omega}{C}$ does not admit a countable Baire-measurable colouring and there is a continuous homomorphism $\varphi: C \to X$ from $\res{\G_0^\omega}$ to $G$. But then $f \circ \varphi$ is a $\sigma(\bm{\Sigma^1_1})$-measurable (and in particular Baire-measurable) countable colouring of $\res{\G_0^\omega}{C}$, a contradiction.
\end{proof}

We begin the proof of \cref{thm:ben-dichotomy} by noting the following equivalent formulations of the second alternative.

\begin{prop}\label{prop:dichotomy-equivalent-conditions}
	Let $X, Y$ be Polish spaces, $E$ a Borel equivalence relation on $X$ and $P \subseteq X \times Y$ an $E$-invariant Borel relation with countable non-empty sections. Then the following are equivalent:
	\begin{enumerate}
		\item[(2)] There is a continuous embedding $\pi_X: 2^\N \times \N \to X$ of $\E_0 \times I_\N$ into $E$ and a continuous injection $\pi_Y: 2^\N \times \N \to Y$ such that for all $x, x' \in 2^\N \times \N$,
		\[\lnot (x \mathrel{\E_0 \times I_\N} x') \implies P_{\pi_X(x)} \cap P_{\pi_X(x')} = \emptyset\]
		and
		\[P_{\pi_X(x)} = \pi_Y([x]_{\E_0 \times I_\N}).\]
		\item[(3)] There is a continuous embedding $\pi_X: 2^\N \to X$ of $\E_0$ into $E$ and a continuous injection $\pi_Y: 2^\N \to Y$ such that for all $x, x' \in 2^\N$,
		\[\lnot(x \mathrel{\E_0} x') \implies P_{\pi_X(x)} \cap P_{\pi_X(x')} = \emptyset\]
		and
		\[\pi_Y(x) \in P_{\pi_X(x)}.\]
		\item[(4)] There is a continuous embedding $\pi_X: 2^\N \to X$ of $\E_0$ into $E$ such that for all $x, x' \in 2^\N$,
		\[\lnot(x \mathrel{\E_0} x') \implies P_{\pi_X(x)} \cap P_{\pi_X(x')} = \emptyset.\]
	\end{enumerate}
\end{prop}

\begin{proof}
	Clearly (2) $\implies$ (3) $\implies$ (4). Assume now that (4) holds, and is witnessed by $\pi_X$. Let $g$ be a uniformization of $P$ and $\pi_Y = g \circ \pi_X$. Since $\pi_Y$ is countable-to-one, by the Lusin--Novikov Theorem there is a Borel non-meager set $B \subseteq 2^\N$ on which $\pi_Y$ is injective and continuous.

	\begin{claim}\label{claim:embed-E0-into-itself}
		There is a continuous embedding of $\E_0$ into $\res{\E_0}{B}$.
	\end{claim}

	\begin{claimproof}
		Fix $s \in 2^{< \N}$ such that $B$ is comeager in $N_s$ and let $U_n$ be a decreasing sequence of open sets that are dense in $N_s$ and satisfy $N_s \cap \bigcap_{n \in \N} U_n \subseteq B$. We will recursively construct sequences $t_n \in 2^{< \N}$ and maps $\varphi_n: 2^n \to 2^{< \N}$ satisfying:
		\begin{enumerate}
			\item $\varphi_0(\emptyset) = s$ and $\varphi_{n+1}(u^\frown i) = \varphi_n(u)^\frown t_n {}^\frown i$ for all $u \in 2^n, i \in 2$; and
			\item $N_{\varphi_n(u)^\frown t_n} \subseteq U_n$ for all $u \in 2^n$.
		\end{enumerate}
		Given this, we define $\varphi(x) = \bigcup_{n \in \N} \varphi_n(\res{x}{n})$ for $x \in 2^\N$. It is clear by (1) that $\varphi: 2^\N \to N_s$ is a continuous embedding of $\E_0$ into $\E_0$, and hence by (2) that $\varphi(x) \in N_s \cap \bigcap_{n \in \N} U_n \subseteq B$ for all $x \in 2^\N$.

		It remains to show how to construct, given $\varphi_n$, a sequence $t_n$ satisfying (2). For any $u \in 2^n$ we have $s \subseteq \varphi_n(u)$ by (1), and as $U_n$ is open and dense in $N_s$ we may extend any $t \in 2^{< \N}$ to some $t' \in 2^{< \N}$ such that $N_{\varphi_n(u)^\frown t'} \subseteq U_n$. Begin with $t = \emptyset$ and recursively extend $t$ to $t'$ in this way for all $u \in 2^n$. Letting $t_n$ be the resulting sequence we have $N_{\varphi_n(u)^\frown t_n} \subseteq U_n$ for all $u \in 2^n$, so (2) is satisfied.
	\end{claimproof}

	Let now $h$ be a continuous embedding of $\E_0$ into $\res{\E_0}{B}$, and note that $\pi_X \circ h, \pi_Y \circ h$ witness that (3) holds.

	Now suppose (3) holds, and is witnessed by $\pi_X, \pi_Y$. Let $h$ be a continuous embedding of $\E_0 \times I_\N$ into $\E_0$, and let $\tilde{\pi}_X = \pi_X \circ h$. Let $F$ be the equivalence relation on $Y$ defined by $y F y'$ iff $y = y'$ or there is some $x\in 2^\N \times \N$ such that $P(\tilde{\pi}_X(x), y)$ and $P(\tilde{\pi}_X(x), y')$. If $y \neq y'$, then the set of $x$ witnessing that $y F y'$ is a single $\E_0 \times I_\N$-class, so by Lusin--Novikov $F$ is Borel. Thus $\pi_Y \circ h$ is an embedding of $\E_0 \times I_\N$ into the countable Borel equivalence relation $F$.

	Because $\pi_Y\circ h$ is injective, its image $A \subseteq Y$ is Borel. Moreover, $\res{F}{A} \cong \E_0 \times I_\N$, so $\res{F}{A}$ is compressible. By \cite[Proposition~2.2]{DJK} there is an isomorphism $f: \res{F}{A} \cong_B \res{F}{[A]_F}$; inspecting the proof of this proposition shows that the map $f$ constructed satisfies $y F f(y)$ for all $y \in A$. (The point is that if $g, h$ are injections with this property, then the function we get by applying proof of Schr\"oder-Bernstein to these maps retains this property.) Thus the map $\tilde{\pi}_Y = f \circ \pi_Y \circ h$ is a Borel embedding of $\E_0 \times I_\N$ into $F$ such that (a) $\tilde{\pi}_Y(2^\N \times \N)$ is $F$-invariant and (b) $\tilde{\pi}_Y(x) F (\pi_Y \circ h)(x)$ for all $x \in 2^\N \times \N$.

	Now $\tilde{\pi}_X, \tilde{\pi}_Y$ would be witnesses to (2), except that $\tilde{\pi}_Y$ is not necessarily continuous. However, $\tilde{\pi}_Y$ is continuous when restricted to an $\E_0 \times I_\N$-invariant comeager Borel set $C$, so it suffices to find a continuous invariant embedding of $\E_0 \times I_\N$ into $\res{(\E_0 \times I_\N)}{C}$. One gets such an embedding by applying \cite[Proposition~1.4]{Mi-dichotomy} to the relation $x R x'$ iff $x (\E_0 \times I_\N) x'$ or $x \notin C$ or $x' \notin C$.
\end{proof}

\begin{rmk}
	From the proof of (3) $\implies$ (2), one sees that if $E$ is a countable Borel equivalence relation then actually one can strengthen (2) so that $\pi_X$ is a continuous invariant embedding of $\E_0 \times I_\N$ into $E$, i.e., a continuous embedding such that additionally $\pi_X([x]_{\E_0 \times I_\N}) = [\pi_X(x)]_E$, for all $x \in 2^\N \times I_\N$.
\end{rmk}

The next two results will be used in the proof of \cref{thm:ben-dichotomy}.

\begin{thm}[\Cref{thm:unif-from-smooth}]\label{thm:unif-from-smooth-4}
	Let $F$ be a smooth Borel equivalence relation on a Polish space $X$, $Y$ be a Polish space, and $P \subseteq X \times Y$ be a Borel set with countable sections. Suppose that
	\[\bigcap_{x \in C} P_x \neq \emptyset\]
	for every $F$-class $C$. Then $P$ admits a Borel $F$-invariant uniformization.
\end{thm}

\begin{proof}
	Let $Z$ be a Polish space and $S: X \to Z$ be a Borel map such that $x F x' \iff S(x) = S(x')$. Define $P^* \subseteq Z \times Y$ by
	\[P^*(z, y) \iff \forall x (S(x) = z \implies P(x, y)).\]
	Note that $P^*$ is $\bm{\Pi^1_1}$, and that if $S(x) = z$ then
	\[P^*_z = \bigcap_{x F x'} P_{x'}\]
	is non-empty and countable.

	By Lusin--Novikov, fix a sequence $g_n$ of Borel maps $g_n: X \to Y$ such that $P = \bigcup_n \operatorname{graph}(g_n)$. Define $Q(x, n) \iff P^*(S(x), g_n(x))$. Then $Q$ is $\bm{\Pi^1_1}$ and has non-empty sections, so by the Novikov Separation Theorem (see \cite[28.5]{CDST} and the following two paragraphs) we can fix a Borel map $h$ uniformizing $Q$.

	Let now $A(z, y) \iff \exists x(S(x) = z \And y = g_{h(x)}(x))$. Then $A \subseteq P^*$ is $\bm{\Sigma^1_1}$, so by the Lusin Separation Theorem there is a Borel set $A \subseteq P^{**} \subseteq P^*$. By \cite[29.19]{CDST}, the set
	\[C = \{z \mid P^{**}_z \text{ is countable}\}\]
	is $\bm{\Pi^1_1}$, and it contains $S(X)$, so by the Lusin Separation Theorem again there is some Borel set $S(X) \subseteq B \subseteq C$.

	By Lusin--Novikov, there is a Borel uniformization $f$ of $R(z, y) \iff B(z) \And P^{**}(z, y)$. Then $f \circ S$ is an $F$-invariant Borel uniformization of $P$.
\end{proof}

\begin{prop}\label{prop:smooth-refinement}
	Let $E$ be an analytic equivalence relation on a Polish space $X$, $F \supseteq E$ be a smooth Borel equivalence relation on $X$, $Y$ be a Polish space, and $P \subseteq X \times Y$ be a Borel $E$-invariant set with countable sections. Suppose that
	\[x F x' \implies P_x \cap P_{x'} \neq \emptyset\]
	for all $x, x' \in X$. Then there is a smooth equivalence relation $E \subseteq F' \subseteq F$ such that
	\[\bigcap_{x \in C} P_x \neq \emptyset\]
	for every $F'$-class $C$.
\end{prop}

\begin{proof}
	Let $G \subseteq X^\N$ be the $\aleph_0$-dimensional hypergraph of $F$-equivalent sequences $x_n$ such that $\bigcap_n P_{x_n} = \emptyset$. By Lusin--Novikov, $G$ is Borel.

	We claim that $\chi_B(G) \leq \aleph_0$. By \cref{thm:upgrade-colouring}, it suffices to show that $G$ has a $\sigma(\bm{\Sigma^1_1})$-measurable countable colouring. Let $S$ be a $\sigma(\bm{\Sigma^1_1})$-measurable selector for $F$ (c.f. \cite[Theorem~6.4.4]{Gao-IDST}) and $g_n$ be a sequence of Borel functions such that $P = \bigcup_n \operatorname{graph}(g_n)$. Then the function $f(x)$ assigning to $x$ the least $n$ such that $P(x, g_n(S(x)))$ is such a colouring. (In fact, $x \mapsto g_{f(x)}(S(x))$ is a $\sigma(\bm{\Sigma^1_1})$-measurable $E$-invariant uniformization of $P$.)

	\begin{claim}\label{claim:enlarge-to-borel}
		Suppose $A$ is analytic and $G$-independent. Then it is contained in an $E$-invariant, $G$-independent Borel set.
	\end{claim}

	\begin{claimproof}
		Let $\Phi$ be a property of subsets of $X$. We say $\Phi$ is \textbf{$\bm{\Pi^1_1}$ on $\bm{\Sigma^1_1}$} if, for every Polish space $Z$ and $\bm{\Sigma^1_1}$ set $P \subseteq Z \times X$, the set $\{z \in Z :  \Phi(P_z)\}$ is $\bm{\Pi^1_1}$. For example, as $G$ is analytic, the property ``$A$ is $G$-independent'' is $\bm{\Pi^1_1}$ on $\bm{\Sigma^1_1}$. We have the following fundamental result about $\bm{\Pi^1_1}$ on $\bm{\Sigma^1_1}$ properties.

		\begin{thm}[First Reflection Theorem {\cite[35.10]{CDST}}]\label{thm:first-reflection-theorem}
			Let $\Phi$ be a $\bm{\Pi^1_1}$ on $\bm{\Sigma^1_1}$ property of subsets of a Polish space $X$. If $A \subseteq X$ is $\bm{\Sigma^1_1}$ and $\Phi(A)$ holds, then there is a Borel set $B \supseteq A$ such that $\Phi(B)$.
		\end{thm}

		It follows that every $G$-independent analytic set is contained in a $G$-independent Borel set. Note also that if $A$ is analytic (resp. $G$-independent) then so is $[A]_E$.

		Let now $A$ be analytic and $G$-independent and let $C' \supseteq [A]_E$ be Borel and $G$-independent. Define
		\[x \in C \iff \forall x' \in X (x E x' \implies x' \in C').\]
		Then $A \subseteq [A]_E \subseteq C \subseteq C'$ so $C$ is $\bm{\Pi^1_1}$, $E$-invariant and $G$-independent.

		By the Lusin Separation Theorem, we may recursively construct an increasing sequence of Borel sets $B_n$ such that $A \subseteq B_0 \subseteq C$ and $[B_n]_E \subseteq B_{n+1} \subseteq C$. Now $B = \bigcup_n B_n$ is Borel, $E$-invariant, $G$-independent (as $B \subseteq C$), and contains $A$.
	\end{claimproof}

	As $\chi_B(G) \leq \aleph_0$, we may fix a sequence $B_n$ of $G$-independent Borel sets covering $X$. By the claim, we may assume that each $B_n$ is $E$-invariant. Define $x F' x' \iff x F x' \And \forall n(x \in B_n \iff x' \in B_n)$. Then $F'$ is a Borel equivalence relation, $E \subseteq F' \subseteq F$, and $F'$ is smooth by \cite[Proposition~5.4.4]{Gao-IDST}. Fix $x = x_0 \in X$, in order to show that
	\[\bigcap_{x F' x'} P_{x'} \neq \emptyset.\]
	Fix an enumeration $y_n, n \geq 1$ of $P_x$, and suppose for the sake of contradiction that this intersection is empty. Then for each $n$, there is some $x_n F' x$ with $y_n \notin P_{x_n}$. Also, $x \in B_k$ for some $k$. But then $x_n \in B_k$ for all $n$, so $B_k$ is not $G$-independent, a contradiction.
\end{proof}

\begin{rmk}
	In \cref{prop:smooth-refinement}, if $F$ admits a Borel selector $S$ then the map $f$ we construct in the proof is a Borel $F$-invariant uniformization of $P$. However, a smooth Borel equivalence relation may not admit a Borel selector \cite[Exercise~5.4.6]{Gao-IDST}.
\end{rmk}

\begin{proof}[Proof of {\cref{thm:ben-dichotomy}}]
	The two cases are mutually exclusive. Indeed, if $\pi_X, \pi_Y$ were maps witnessing (2) and $f$ were a Borel $E$-invariant uniformization of $P$ then $f \circ \pi_X$ would be a Borel reduction of $\E_0 \times I_\N$ to $\Delta_Y$, which is impossible as $\E_0 \times I_\N$ is not smooth. To see that at least one of them holds, define the graph $G$ on $X$ by $x G x' \iff P_x \cap P_{x'} = \emptyset$. By Lusin--Novikov, this is a Borel graph. By the $(\G_0, \bbH_0)$ dichotomy applied to the pair $(G, E)$, exactly one of the following two cases hold.

	\textbf{Case 1:} There is a smooth Borel equivalence relation $F \supseteq E$ such that $\chi_B(G \cap F) \leq \aleph_0$. Note that if $A$ is $(G \cap F)$-independent then so is $[A]_E$ so, by the proof of \cref{claim:enlarge-to-borel}, every analytic $(G \cap F)$-independent set is contained in a Borel, $E$-invariant, $(G \cap F)$-independent set.

	We may therefore fix a sequence $B_n$ of $E$-invariant, $(G \cap F)$-independent Borel sets covering $X$, and define $x F' x' \iff x F x' \And \forall n (x \in B_n \iff x' \in B_n)$. Then $F'$ is a Borel equivalence relation, $E \subseteq F' \subseteq F$, $F'$ is smooth (by \cite[Proposition~5.4.4]{Gao-IDST}), and each $F'$-class is contained in some $B_n$, hence is $G$-independent. That is, if $x F' x'$ then $P_x \cap P_{x'} \neq \emptyset$. Thus case (1) of \cref{thm:ben-dichotomy} holds, by \cref{prop:smooth-refinement,thm:unif-from-smooth-4}.

	\textbf{Case 2:} There is a continuous homomorphism $f: 2^\N \to X$ from $(\G_0, \bbH_0)$ to $(G, E)$. We claim that (2) holds in \cref{thm:ben-dichotomy}, and for this it suffices to show that (4) holds in \cref{prop:dichotomy-equivalent-conditions}. To see this, consider $F = (f \times f)^{-1}(E), R = (f \times f)^{-1}(G)$. Then $\bbH_0 \subseteq F$ and each $F$-section is $\G_0$-independent, hence meager, so by Kuratowski--Ulam $F$ is meager. We claim $R$ is comeager. To see this, fix $x \in 2^\N$ and consider $R^c_x = \{x' : P_{f(x)} \cap P_{f(x')} \neq \emptyset\}$. Fix an enumeration $y_n$ of $P_{f(x)}$, and let $A_n = \{x' : y_n \in P_{f(x')}\}$. Then each $A_n$ is $\G_0$-independent, hence meager, and $R^c_x = \bigcup_n A_n$. Thus $R$ has comeager sections, and by Kuratowski--Ulam $R$ is comeager. One can now recursively construct a continuous homomorphism $g$ from $((\Delta_{2^\N})^c, \E_0^c, \E_0)$ to $((f \times f)^{-1}(\Delta_X)^c, R, \E_0)$; see e.g. the proof of \cite[Proposition~11]{Mi3}. Then $f \circ g$ satisfies (4).
\end{proof}

\subsection{An \texorpdfstring{$\aleph_0$}{aleph0}-dimensional \texorpdfstring{$(\G_0, \bbH_0)$}{(G0, H0)} dichotomy}\label{sec:aleph-0-dim-g0-h0}

In this section we state and prove an $\aleph_0$-dimensional analogue of Miller's $(\G_0, \bbH_0)$ dichotomy \cite[Theorem~25]{Mi-survey}. This dichotomy generalizes Lecomte's $\aleph_0$-dimensional $\G_0$ dichotomy \cite{L} (see also \cite{Mi-countable-dim}) in the same way that Miller's $(\G_0, \bbH_0)$ dichotomy generalizes the $\G_0$ dichotomy \cite{KST}. We then state an effective analogue of this theorem, and indicate the changes that must be made to prove it.

For $s \in \N^{< \N}$, define a Borel directed hypergraph on $\N^\N$ by
\[\G_s^\omega = \{(s^\frown i^\frown x)_{i \in \N} : x \in \N^\N\},\]
and write $\G_S^\omega = \bigcup_{s \in S} \G_s^\omega$ for $S \subseteq \N^{< \N}$. For $t = (t_0, t_1) \in \bigcup_{n \in \N} \N^n \times \N^n$, define a Borel directed graph on $2^\N$ by
\[\bbH_t^\omega = \{(t_0 {}^\frown 0^\frown x, t_1 {}^\frown 1^\frown x) : x \in \N^\N\},\]
and write $\bbH_T^\omega = \bigcup_{t \in T} \bbH_t^\omega$ for $T \subseteq \bigcup_{n \in \N} \N^n \times \N^n$.

Given $S \subseteq \N^{< \N}$ and $T \subseteq \bigcup_{n \in \N} \N^n \times \N^n$, we say $(S, T)$ is \tb{sparse} if
\[|S \cap \N^n| + |T \cap \N^n \times \N^n| \leq 1\]
for all $n \in \N$, and $(S, T)$ is \tb{full} if
\[|S \cap \N^n| + |T \cap \N^n \times \N^n| \geq 1\]
for all $n \in \N$. We say $(S, T)$ is \tb{dense} if both $S$ and $T$ are dense, i.e.,
\[\forall u \in \N^{< \N} \exists s \in S (u \subseteq s) ~~\text{and}~~ \forall u \in \N^{< \N} \times \N^{< \N} \exists t \in T \forall i < 2 (u_i \subseteq t_i).\]

Fix a strictly increasing sequence $\alpha \in \N^\N$ and a pair $(S, T)$ that is sparse, full and dense. Let
\[X_\alpha = \{x \in \N^\N : \forall n \exists m \geq n (\res{x}{m} \in \alpha(m)^m)\},\]
and note that $X_\alpha$ is a dense $G_\delta$ set in $\N^\N$. We define $\G_0^\omega = \res{\G_S^\omega}{X_\alpha}$ and $\bbH_0^\omega = \res{\bbH_T^\omega}{X_\alpha}$.

\begin{prop}[{\cite[Lemma~2.1]{L}}]\label{prop:G0-independent-meager}
	Let $A \subseteq X_\alpha$ be Baire measurable and $\G_0^\omega$-independent. Then $A$ is meager.
\end{prop}

\begin{proof}
	Suppose $A$ is non-meager, and fix an open set $N_s = \{x \in \N^\N : s \subseteq x\}$ in which $A$ is comeager. By density of $S$, we may assume wlog that $s \in S$. For each $n$, the set $A_n = \{x \in \N^\N : s^\frown n^\frown x \in A\}$ is comeager, so there is some $x \in \bigcap_n A_n$. But then $x_n = s^\frown n^\frown x \in A$ and $\G_0^\omega((x_n))$, so $A$ is not $\G_0^\omega$-independent.
\end{proof}

Let $R$ be a quasi-order on a Polish space $X$. We let $\equiv_R$ denote the equivalence relation $x \equiv_R y \iff x R y \And y R x$. We say $R$ is \textbf{lexicographically reducible} if there is a Borel reduction of $R$ to the lexicographic order $\leq_{\text{lex}}$ on $2^\alpha$, for some $\alpha < \omega_1$. In particular, if $R$ is lexicographically reducible then $R$ is Borel and total and $\equiv_R$ is smooth.

Given a quasi-order $R$ on $X$ and $A \subseteq X$, we define
\[[A]^R = \{y : \exists x \in A(x R y)\} ~~\text{and}~~ [A]_R = \{y : \exists x \in A(y R x)\}\]
and say $A$ is \tb{closed upwards} (resp. \tb{downward}) for $R$ if $A = [A]^R$ (resp. $A = [A]_R$). If $A, B \subseteq X$, we say $(A, B)$ is \tb{$R$-independent} if $A \times B \cap R = \emptyset$.

\begin{prop}[{\cite[Proposition~5]{Mi4}}]\label{prop:H0-ergodic}
	Let $R$ be a total quasi-order on $X_\alpha$ with $\bbH_0^\omega \subseteq R$. If $R$ is Baire measurable, then some $\equiv_R$-class is comeager.
\end{prop}

\begin{proof}
	We will show that $R$ is comeager. From this it follows that $\mathord{\equiv_R} = R \cap R^{-1}$ is comeager, and hence by Kuratowski--Ulam there is a comeager $\equiv_R$-class.

	As $R$ is total, $X^2 = R \cup R^{-1}$ so $R$ is non-meager. Let $u_0, u_1 \in \N^{< \N}$ be such that $R$ is comeager in $N_{u_0} \times N_{u_1}$. We will show that $R$ is non-meager in $N_{v_0} \times N_{v_1}$ for all $v_0, v_1 \in \N^{< \N}$. To see this, let $t, t' \in T$ be such that $v_0 \subseteq t_0, u_0 \subseteq t_1, u_1 \subseteq t'_0, v_1 \subseteq t'_1$. The set of all $(x, y) \in \N^\N \times \N^\N$ such that $t_1 {}^\frown 1^\frown x R t'_0 {}^\frown 0^\frown y$ is comeager, and for any such $(x, y)$ we have
	\[t_0 {}^\frown 0^\frown x R t_1 {}^\frown 1^\frown x R t'_0{}^\frown 0^\frown y R t'_1 {}^\frown 1^\frown y.\]
	Thus $R$ is non-meager in $N_{v_0} \times N_{v_1}$.
\end{proof}

\begin{prop}[{\cite[Proposition~1]{Mi4}}]\label{prop:reflection-quasi-order}
	Let $R$ be an analytic quasi-order on a Polish space $X$ and $A_0, A_1 \subseteq X$ be analytic such that $(A_0, A_1)$ is $R$-independent. Then there are Borel sets $A_i \subseteq B_i$ such that $(B_0, B_1)$ is $R$-independent, $B_0$ is closed upwards for $R$, and $B_1$ is closed downwards for $R$.
\end{prop}

\begin{proof}
	By \cref{thm:first-reflection-theorem}, if $A_0, A_1$ are analytic and $(A_0, A_1)$ is $R$-independent then there are Borel sets $B_i \supseteq A_i$ such that $(B_0, B_1)$ is $R$-independent. Indeed, if $A_1, R$ are analytic then the property (of $A \subseteq X$) that ``$(A, A_1)$ is $R$-independent'' is $\bm{\Pi^1_1}$ on $\bm{\Sigma^1_1}$, so there is a Borel set $B_0 \supseteq A_0$ such that $(B_0, A_1)$ is $R$-independent, and a symmetric argument gives a Borel set $B_1 \supseteq A_1$ such that $(B_0, B_1)$ is $R$-independent.

	Note that if $A_0, A_1$ are analytic and $(A_0, A_1)$ is $R$-independent, then the same holds for $[A_0]^R, [A_1]_R$. We may therefore recursively construct sequences $B_i^n$ of Borel sets such that $A_i \subseteq B_i^0$, $[B_0^n]^R \subseteq B_0^{n+1}$, $[B_1^n]_R \subseteq B_1^{n+1}$, and $(B_0^n, B_1^n)$ is $R$-independent. It is now easy to verify that the sets $B_i = \bigcup_n B_i^n$ are Borel, $(B_0, B_1)$ is $R$-independent, $A_i \subseteq B_i$, $B_0$ is closed upward for $R$, and $B_1$ is closed downward for $R$.
\end{proof}

Let $F$ be an equivalence relation on $X$ and $G$ be an $\aleph_0$-dimensional directed hypergraph on $X$. We say $A \subseteq X$ is \tb{$F$-locally $G$-independent} if there is no sequence $x_n \in A$ of pairwise $F$-equivalent points with $(x_n) \in G$. If $X$ is Polish, a \tb{countable Borel $F$-local colouring of $G$} is a sequence of $F$-locally $G$-independent Borel sets that cover $X$.

\begin{thm}\label{thm:countable-dimensional-dichotomy}
	Let $G$ be an analytic $\aleph_0$-dimensional directed hypergraph on a Polish space $X$ and $R$ an analytic quasi-order on $X$. Then exactly one of the following holds:
	\begin{enumerate}[label=(\arabic*)]
		\item There is a lexicographically reducible quasi-order $R'$ on $X$ such that $R \subseteq R'$ and there is a countable Borel $\equiv_{R'}$-local colouring of $G$.
		\item There is a continuous homomorphism from $(\G_0^\omega, \bbH_0^\omega)$ to $(G, R)$.
	\end{enumerate}
\end{thm}

\begin{proof}
	To see that (1) and (2) are mutually exclusive it suffices to show that there is no Borel total quasi-order $R'$ on $X_\alpha$ such that $\bbH_0^\omega \subseteq R'$ and there is a countable Borel $\equiv_{R'}$-local colouring of $\G_0^\omega$. Towards a contradiction, let $R'$ be such a quasi-order. By assumption, if $C$ is an $\equiv_{R'}$-class then there is a countable Borel colouring of $\res{\G_0^\omega}{C}$, so $C$ is meager by \cref{prop:G0-independent-meager}. But there is a comeager $\equiv_{R'}$-class by \cref{prop:H0-ergodic}, a contradiction.

	We now show that at least one of these alternatives hold. Fix continuous maps $\pi_G: \N^\N \to X^\N$, $\pi_R: \N^\N \to X^2$ such that $G = \pi_G(\N^\N)$, $R = \pi_R(\N^\N)$. Let $d$ denote the usual metric on $\N^\N$, and $d_X$ be a complete metric compatible with the Polish topology on $X$.

	Let $V$ be a set, $H_0$ be an $\aleph_0$-dimensional directed hypergraph on $V$ with edge set $E_0$, and $H_1$ be a directed graph on $V$ with edge set $E_1$. A \textbf{copy} of $(H_0, H_1)$ in $(G, R)$ is a triple $\varphi = (\varphi_X, \varphi_G, \varphi_R)$ where $\varphi_X: V \to X, \varphi_G: E_0 \to \N^\N, \varphi_R: E_1 \to \N^\N$, such that
	\[e = (v_n) \in E_0 \implies \pi_G(\varphi_G(e)) = (\varphi_X(v_n))_{n \in \N},\]
	and
	\[e = (v, u) \in E_1 \implies \pi_R(\varphi_R(e)) = (\varphi_X(v), \varphi_X(u)).\]
	That is, $\varphi_X$ is a homomorphism from $(H_0, H_1)$ to $(G, R)$, and $\varphi_G, \varphi_R$ pick out elements of $\N^\N$ that witness this (via $\pi_G, \pi_R$).

	Let $\hom(H_0, H_1; G, R)$ denote the set of all copies of $(H_0, H_1)$ in $(G, R)$. Note that if $V, E_0, E_1$ are countable, then $\hom(H_0, H_1; G, R) \subseteq X^V \times (\N^\N)^{E_0} \times (\N^\N)^{E_1}$ is closed, and hence Polish.

	We let $\bullet$ denote the trivial (hyper)graph, i.e., the (hyper)graph $H$ with a single vertex and edge set $E = \emptyset$. In this case, a copy of $(\bullet, \bullet)$ in $(G, R)$ is simply an element of $X$, and we identify $\hom(\bullet, \bullet; G, R)$ with $X$.

	Suppose now we have $H_0, H_1$ as above, with $V, E_0, E_1$ countable, and consider $\H \subseteq \hom(H_0, H_1; G, R)$. Let $\H(v) = \{\varphi_X(v) : \varphi \in \H\} \subseteq X$ for $v \in V$, and note that $\H(v)$ is analytic whenever $\H$ is analytic. Define $\H(e) \subseteq \N^\N$ similarly for $e \in E_0 \cup E_1$. Now call $\H$ \textbf{tiny} if it is Borel and there is a lexicographically reducible quasi-order $R'$ on $X$ such that $R \subseteq R'$ and one of the following holds:
	\begin{enumerate}[label=(\arabic*)]
		\item $\H(v)$ is $\equiv_{R'}$-locally $G$-independent for some $v \in V$.
		\item $\forall \varphi \in \H \exists u, v \in V(\varphi_X(u) \not\equiv_{R'} \varphi_X(v))$.
	\end{enumerate}
	In this case, we call $R'$ a \textbf{witness} that $\H$ is tiny, and say $\H$ is tiny of type 1 (resp. 2) if $\H, R'$ satisfy (1) (resp. (2)). Finally, we say $\H$ is \textbf{small} if it is in the $\sigma$-ideal generated by the tiny sets, and otherwise we call $\H$ \textbf{large}.

	Finally, fix $H_0, H_1$ as above with $V, E_0, E_1$ countable. For $v \in V$ we define the $\aleph_0$-dimensional directed hypergraph $\oplus_v H_0$ (resp. the directed graph $\oplus_v H_1$) by taking a countable disjoint union of $H_0$ (resp. $H_1$) on the vertex set $V \times \N$, and adding the edge $((v, n))_{n \in \N}$ to $\oplus_v H_0$; note that we do not add any new edges to $\oplus_v H_1$. Similarly, for $u, v \in V$ we define the $\aleph_0$-dimensional directed hypergraph $H_0 \prescript{}{u}{+}^{}_v H_0$ (resp. the directed graph $H_1 \prescript{}{u}{+}^{}_v H_1$) by taking a countable disjoint union of $H_0$ (resp. $H_1$) on the vertex set $V \times \N$, and adding the edge $((u, 0), (v, 1))$ to $H_1 \prescript{}{u}{+}^{}_v H_1$; we do not add any new edges to $H_0 \prescript{}{u}{+}^{}_v H_0$. Note that there are natural continuous projection maps
	\[\hom(\oplus_v H_0, \oplus_v H_1; G, R) \to \hom(H_0, H_1; G, R)\]
	and
	\[\hom(H_0 \prescript{}{u}{+}^{}_v H_0, H_1 \prescript{}{u}{+}^{}_v H_1; G, R) \to \hom(H_0, H_1; G, R),\]
	for all $n \in \N$, taking $\varphi$ to its restriction $\varphi^n$ to $V \times \{n\}$. If $\H \subseteq \hom(H_0, H_1; G, R)$, we let
	\begin{align*}
		\oplus_v \H &= \{ \varphi \in \hom(\oplus_v H_0, \oplus_v H_1; G, R) : \forall n(\varphi^n \in \H)\},\\
		\H \prescript{}{u}{+}^{}_v \H &= \{ \varphi \in \hom(H_0 \prescript{}{u}{+}^{}_v H_0, H_1 \prescript{}{u}{+}^{}_v H_1; G, R) : \forall n(\varphi^n \in \H)\}.
	\end{align*}
	Note that if $\H$ is Borel then so are $\oplus_v \H$ and $\H \prescript{}{u}{+}^{}_v \H$.

	\begin{claim}\label{claim:dichotomy-small}
		If $\hom(\bullet, \bullet; G, R)$ is small, then there is a lexicographically reducible quasi-order $R'$ on $X$ such that $R \subseteq R'$ and there is a countable Borel $\equiv_{R'}$-local colouring of $G$.
	\end{claim}

	\begin{claimproof}
		As noted above, we may identify $\hom(\bullet, \bullet; G, R)$ with $X$. As the trivial (hyper)graph has only one vertex, there cannot be any tiny sets $\H \subseteq \hom(\bullet, \bullet; G, R)$ of type $2$. We therefore have by our assumption a sequence $\H_n$ of Borel sets covering $X$ such that each $\H_n$ is tiny of type $1$, i.e., there is a lexicographically reducible quasi-order $R_n$ such that $\H_n$ is $\equiv_{R_n}$-locally $G$-independent.

		Let $f_n: X \to 2^{\alpha_n}$ be a Borel reduction of $R_n$ to the lexicographic ordering on $2^{\alpha_n}$, $\alpha_n < \omega_1$. Let $\alpha = \sum_n \alpha_n < \omega_1$ and define $f: X \to 2^\alpha$ by $f(x) = f_0(x)^\frown f_1(x)^\frown f_2(x)^\frown \cdots$. As $f$ is Borel, the quasi-order $x R' y \iff f(x) \leq_{\text{lex}} f(y)$ is lexicographically reducible. Moreover, $R \subseteq \bigcap_n R_n \subseteq R'$. Finally, $\mathord{\equiv_{R'}} = (\bigcap_n \equiv_{R_n})$, so each $\H_n$ is $\equiv_{R'}$-locally $G$-independent.
	\end{claimproof}

	\begin{claim}\label{claim:dichotomy-diameter}
		Let $H_0, H_1$ be as above with $V, E_0, E_1$ countable, $F \subseteq V \cup E_0 \cup E_1$ be finite, $\varepsilon > 0$, and $\H \subseteq \hom(H_0, H_1; G, R)$ be large and Borel. Then there is a large Borel set $\H' \subseteq \H$ for which $\diam_{d_X}(\H'(v)) < \varepsilon$ for all $v \in F \cap V$ and $\diam_{d}(\H'(e)) < \varepsilon$ for all $e \in F \cap (E_0 \cup E_1)$.
	\end{claim}

	\begin{claimproof}
		By induction it suffices to consider the case $|F| = 1$. We will show how to handle the case where $F = \{v\} \subseteq V$; the case where $F \subseteq E_0 \cup E_1$ is handled identically.

		Fix now $v \in V$ and let $U_n$ be a sequence of open sets covering $X$ such that $\diam_{d_X}(U_n) < \varepsilon$ for all $n \in \N$. For $n \in \N$ let
		\[\H_n = \{\varphi \in \H : \varphi_X(v) \in U_n\},\]
		and note that $\H_n$ is Borel and $\H_n(v) \subseteq U_n$ so $\diam_{d_X}(\H_n(v)) < \varepsilon$. If each $\H_n$ is small then $\H = \bigcup_n \H_n$ is small, as the small sets form a $\sigma$-ideal, so there is some $n$ for which $\H_n$ is large and we can take $\H' = \H_n$.
	\end{claimproof}

	\begin{claim}\label{claim:dichotomy-large}
		Let $H_0, H_1$ be as above with $V, E_0, E_1$ countable, and suppose $\H \subseteq \hom(H_0, H_1; G, R)$ is Borel and large. Then $\oplus_v \H, \H \prescript{}{u}{+}^{}_v \H$ are Borel and large.
	\end{claim}

	\begin{claimproof}
		That these sets are Borel is clear.

		Now suppose $\oplus_v \H$ is small, write $\oplus_v \H = \bigcup_{i \in 2, n \in \N} \F^i_n$ with $\F^i_n$ tiny of type $i+1$, and fix witnesses $R^i_n$ that $\F^i_n$ is tiny of type $i+1$. Arguing as in the proof of \cref{claim:dichotomy-small}, there is a lexicographically reducible quasi-order $R'$ with $R \subseteq R'$ and $\mathord{\equiv_{R'}} = \bigcap_{i, n} \equiv_{R^i_n}$, and hence we may assume wlog that $R^i_n = R'$ for all $i, n$.

		Let $v_n \in V$ be such that $\F^0_n(v_n)$ is $\equiv_{R'}$-locally $G$-independent. By \cref{thm:first-reflection-theorem}, we may fix Borel sets $\F^0_n(v_n) \subseteq A_n$ which are $\equiv_{R'}$-locally $G$-independent. Define $\H_n = \{\varphi \in \H : \varphi_X(v_n) \in A_n\}$ and let
		\[\H' = \H \setminus \left(\{\varphi \in \H : \exists u, w \in V (\varphi_X (u) \not\equiv_{R'} \varphi_X(w))\} \cup \bigcup_n \H_n \right).\]

		We claim $\H'$ is tiny, which implies that $\H$ is small. Clearly $\H'$ is Borel, and we claim $\H'(v)$ is $\equiv_{R'}$-locally $G$-independent. Indeed, if $\varphi_n \in \H'$ and $G(((\varphi_n)_X(v))_{n \in \N})$, then there is some $\varphi \in \oplus_v \H$ with $\varphi^n = \varphi_n$ for all $n$. But then $\varphi \in \F^1_n$ for some $n$, so there are $u, w \in V \times \N$ such that $\varphi_X(u) \not\equiv_{R'} \varphi_X(w)$. Since $\varphi^n \in \H'$ for all $n$, we may assume that $u = (v, i), w = (v, j)$ for some $i \neq j$. But then $(\varphi_i)_X(v) = \varphi^i_X(v) \not\equiv_{R'} \varphi^j_X(v) = (\varphi_j)_X(v)$.

		Next suppose $\H \prescript{}{u}{+}^{}_v \H$ is small, write $\H \prescript{}{u}{+}^{}_v \H = \bigcup_{i \in 2, n \in \N} \F^i_n$ with $\F^i_n$ tiny of type $i+1$, and fix witnesses $R^i_n$ that $\F^i_n$ is tiny of type $i+1$. As before, we may assume $R^i_n = R'$ for a single $R'$ and we define $\H_n, \H'$ in the same way, so that it suffices again to show that $\H'$ is tiny of type 2.

		Let $\varphi_i \in \H', i \in 2$, and suppose $(\varphi_0)_X(u) R (\varphi_1)_X(v)$. Then there is some $\varphi \in \H \prescript{}{u}{+}^{}_v \H$ with $\varphi^0 = \varphi_0$ and $\varphi^i = \varphi_1$ for $i > 0$. As before, we find that we must have $\varphi_X((u, 0)) \not\equiv_{R'} \varphi_X((v, 1))$, so that $(\varphi_0)_X(u) \not\equiv_{R'} (\varphi_1)_X(v)$. Thus, $(\H'(u), \H'(v))$ is $(R \cap \mathord{\equiv_{R'}})$-independent, and by \cref{prop:reflection-quasi-order} we can find Borel sets $\H'(u) \subseteq A, \H'(v) \subseteq B$ such that $A$ is closed upwards for $R \cap \mathord{\equiv_{R'}}$, $B$ is closed downwards for $R \cap \mathord{\equiv_{R'}}$, and $(A, B)$ is $(R \cap \mathord{\equiv_{R'}})$-independent. Then
		\[x Q y \iff x R' y \And (x \equiv_{R'} y \And x \in A \implies y \in A)\]
		is a lexicographically reducible quasi-order containing $R$, and $\H'$ is tiny of type 2 with witness $Q$.
	\end{claimproof}

	If $\hom(\bullet, \bullet; G, R)$ is small, then alternative (1) holds by \cref{claim:dichotomy-small}. Suppose now that $\hom(\bullet, \bullet; G, R)$ is large. We define a sequence $G_n$ of $\aleph_0$-dimensional directed hypergraphs on $\N^n$ and a sequence $H_n$ of directed graphs on $\N^n$ as follows:
	\begin{align*}
		G_n(x_i) \iff& \exists k < n \, \exists s \in (S \cap \N^k) \, \exists u \in \N^{n-k-1} \, \forall i(x_i = s^\frown i^\frown u),\\
		x H_n y \iff& \exists k < n \, \exists (t_0, t_1) \in (T \cap \N^k \times \N^k) \\
		&\qquad\qquad \exists u \in \N^{n-k-1}(x = t_0^\frown 0^\frown u \And y = t_1^\frown 1^\frown y).
	\end{align*}
	In particular, $G_0, H_0$ are the trivial graphs on $V = \{\emptyset\}$, so $\hom(G_0, H_0; G, R) = \hom(\bullet, \bullet; G, R)$ is large.

	As $(S, T)$ is sparse and full, exactly one of $S \cap \N^n$ and $T \cap \N^n \times \N^n$ is non-empty for any $n \in \N$. It follows that for any $n \in \N$, if there exists $s \in S \cap \N^n$ then $G_{n+1} = \oplus_s G_n$ and $H_{n+1} = \oplus_s H_n$, and if there exists $(t_0, t_1) \in T \cap \N^n \times \N^n$ then $G_{n+1} = G_n \prescript{}{t_0}{+}^{}_{t_1} G_n$ and $H_{n+1} = H_n \prescript{}{t_0}{+}^{}_{t_1} H_n$; moreover, exactly one of these cases holds. Therefore the graphs $(G_{n+1}, H_{n+1})$ can always be constructed from $(G_n, H_n)$ via these operations. We note also that
	\[\G_0^\omega ((x_i)_{i \in \N}) \iff \exists N \forall n \geq N (G_n((\res{x_i}{n})_{i \in \N}))\]
	and
	\[x \bbH_0^\omega y \iff \exists N \forall n \geq N (\res{x}{n} \mathop{H_n} \res{y}{n}),\]
	and $G_n, H_n$ have countably many vertices and edges.

	Fix a large Borel set $\H_0 \subseteq \hom(\bullet, \bullet; G, R)$ with $\diam_{d_X}(\H_0(\bullet)) < 1$; such a set exists by \cref{claim:dichotomy-diameter} as $\hom(\bullet, \bullet; G, R)$ is large. By \cref{claim:dichotomy-diameter,claim:dichotomy-large}, we can recursively construct a sequence of large Borel sets $\H_n \subseteq \hom(G_n, H_n; G, R)$ such that $\H_{n+1} \subseteq \H_n \oplus_s \H_n$ for $s \in S \cap \N^n$, $\H_{n+1} \subseteq \H_n \prescript{}{t_0}{+}^{}_{t_1} \H_n$ for $(t_0, t_1) \in T \cap \N^n \times \N^n$, $\diam_{d_X}(\H_n(x)) < 2^{-n}$ for all $x \in \alpha(n)^n$, and $\diam_d(\H(e)) < 2^{-n}$ for all $e \in G_n \cup H_n$ with $e_0 \in \alpha(n)^n$, where $e_0$ denotes the first vertex in $e$ (note that for every $n$ there are only finitely many such edges). It follows that $\{f(x)\} = \bigcap_n \overline{\H_n(\res{x}{n})}$ exists and is well defined for $x \in X_\alpha$, and that this map $f: X_\alpha \to X$ is continuous. To see that it is a homomorphism of $\G_0^\omega$ to $G$, suppose $\G_0^\omega((x_i)_{i \in \N})$ and let $N$ be sufficiently large that $G_N((\res{x_i}{N})_{i \in \N})$. Then $\{y\} = \bigcap_{n \geq N} \overline{\H_n((\res{x_i}{n})_{i \in \N})}$ exists and is well defined, and by continuity of $\pi_G$ we have $(f(x_i))_{i \in \N} = \pi_G(y) \in G$. A similar argument shows that $f$ is a homomorphism from $\bbH_0^\omega$ to $R$.
\end{proof}

\begin{rmk}
	The assumption that $(S, T)$ is full in the definition of $\G_0^\omega, \bbH_0^\omega$ is not necessary, and was simply done for notational convenience in the proof of \cref{thm:countable-dimensional-dichotomy}. If $(S, T)$ were not full, then we would sometimes need to take the graphs $G_{n+1}, H_{n+1}$ to be countable disjoint unions of the graphs $G_n, H_n$ (without adding any edges), and \cref{claim:dichotomy-large} would need to be extended to accommodate this (the extension would follow by essentially the same argument).

	An inspection of the proof shows that at most one of the alternatives holds when $(S, T)$ is dense, and at least one of them holds when $(S, T)$ is sparse. Moreover, if $(S, T)$ is sparse and $(S', T')$ is dense then there is a continuous embedding of $(\G_S^\omega, \bbH_T^\omega)$ into $(\G_{S'}^\omega, \bbH_{T'}^\omega)$ (such an embedding can be constructed directly, without appealing to \cref{thm:countable-dimensional-dichotomy}). Therefore the choice of a sparse, dense pair $(S, T)$ in the definition of $(\G_0^\omega, \bbH_0^\omega)$ does not matter up to continuous bi-embedability.
\end{rmk}

This dichotomy admits the following effective refinement:

\begin{thm}\label{thm:effective-aleph0-dimensional-dichotomy}
	Let $G$ be a $\Sigma^1_1$ $\aleph_0$-dimensional directed hypergraph on a Polish space $X$, and $R$ a $\Sigma^1_1$ partial order on $X$. Then exactly one of the following holds:
	\begin{enumerate}
		\item There is a quasi-order $R'$ on $X$ such that $R \subseteq R'$, there is a countable $\Delta^1_1$ $\equiv_{R'}$-local colouring of $G$, and there is a $\Delta^1_1$ reduction of $R'$ to the lexicographic order $\leq_{\text{lex}}$ on $2^\alpha$, for some $\alpha < \omega_1^{CK}$.

		\item There is a continuous homomorphism from $(\G_0^\omega, \bbH_0^\omega)$ to $(G, R)$.
	\end{enumerate}
\end{thm}

To prove this, we make the following modifications to the proof of \cref{thm:countable-dimensional-dichotomy}. First, we choose $\pi_G, \pi_H$ to be computable (restricting their domains appropriately to $\Pi^0_1$ sets). We then replace ``Borel'' with ``$\Delta^1_1$'' and ``lexicographically reducible'' with ``admitting a $\Delta^1_1$ reduction to $\leq_{\text{lex}}$ on $2^\alpha$, for some $\alpha < \omega_1^{CK}$'' in the definition of tiny sets.

We now have the following:

\begin{lem}\label{lem:effective-aleph-0-dimensional-dichotomy-key-lemma}
	Let $V$ be a set, $H_0$ be an $\aleph_0$-dimensional directed hypergraph on $V$ with edge set $E_0$ and $H_1$ be a directed graph on $V$ with edge set $E_1$, with $V, E_0, E_1$ countable. Suppose $\H \subseteq \hom(H_0, H_1; G, R)$ is small and $\Delta^1_1$. Then one can find:

	(1) a uniformly $\Delta^1_1$ sequence of tiny sets $(\F^i_n)_{i \in 2, n \in \N}$ covering $\H$,

	(2) a uniformly $\Delta^1_1$ sequence $(R^i_n)_{i \in 2, n \in \N}$ of quasi-orders on $\N$,

	(3) a uniformly $\Delta^1_1$ sequence of ordinals $\alpha^i_n < \omega_1^{CK}$,

	(4) a uniformly $\Delta^1_1$ sequence $(f^i_n)_{i \in 2, n \in \N}$ of maps $f^i_n: \N^\N \to 2^{\alpha^i_n}$, and

	(5) a uniformly $\Delta^1_1$ sequence $v_n \in V$,

	\noindent such that the sets $\F^i_n$ are pairwise disjoint, $R \subseteq R^i_n$ for all $i, n$, each $f^i_n$ is a reduction of $R^i_n$ to $\leq_{\text{lex}}$ on $2^{\alpha^i_n}$, $\F^0_n(v_n)$ are $\equiv_{R^0_n}$-locally $G$-independent, and $\forall \varphi \in \H \exists u, v \in V (\varphi_X(u) \not\equiv_{R^1_n} \varphi_X(v))$.
\end{lem}

\begin{proof}[Proof sketch]
	Fix a nice coding $D \ni n \mapsto D_n$ of the $\Delta^1_1$ sets (c.f. \cref{subsec:equiv-effective-consequences}). The property of the tuple $(\F, R', \alpha, f, v, i)$ that ``$f: X \to 2^\alpha$ is a reduction of $R'$ to $\leq_{\text{lex}}$ and $R'$ is a witness that $\F$ is tiny of type $i$ (via $v \in V$ if $i = 1$)'' is $\Pi^1_1$ on $\Delta^1_1$. It follows that the relation $P(\varphi, n) \iff$ ``$\varphi \notin \H$ or $n \in D$ codes such a tuple with $\varphi \in \F$'' is $\Pi^1_1$, and hence by the Number Uniformization Theorem for $\Pi^1_1$ \cite[4B.4]{Mo} there is a $\Delta^1_1$ map $g: \H \to D$ taking each $\varphi \in \H$ to the code of such a tuple. The image of $\H$ under $g$ is $\Sigma^1_1$, and the set of all $n \in D$ coding tuples as above is $\Pi^1_1$, so by $\Sigma^1_1$ separation there is a $\Delta^1_1$ set $A \subseteq D$ containing $g(\H)$ and such that every element in $A$ codes a tuple as above. One can then fix a $\Delta^1_1$ enumeration of $A$, which satisfies all of the above conditions except maybe pairwise disjointness of the family $\F^i_n$, and this can be fixed by a straightforward recursive construction.
\end{proof}

The effective analogue of \cref{claim:dichotomy-small} follows immediately. We note that \cref{thm:first-reflection-theorem} is effective enough that the proof of \cref{prop:reflection-quasi-order} is effective as well. \Cref{claim:dichotomy-large} then follows using \cref{lem:effective-aleph-0-dimensional-dichotomy-key-lemma}. The rest of the proof is identical to that of \cref{thm:countable-dimensional-dichotomy}.

\subsection{Proof of \texorpdfstring{\cref{thm:ben-dichotomy}}{Theorem~\ref*{thm:ben-dichotomy}} from the \texorpdfstring{$\aleph_0$}{aleph0}-dimensional \texorpdfstring{$(\G_0, \bbH_0)$}{(G0, H0)} dichotomy}\label{sec:proof-of-ben-2}

\begin{proof}[Proof of \cref{thm:ben-dichotomy}]
	We have shown in the first paragraph of the proof of \cref{thm:ben-dichotomy} given at the end of \cref{sec:proof-of-ben-1} that the two alternatives are mutually exclusive (we note that this argument used only the fact that $\E_0 \times I_\N$ is not smooth, and in particular did not rely on any graph-theoretic dichotomies). To see that at least one of them holds, define the $\aleph_0$-dimensional hypergraph $G$ on $X$ by $G(x_n) \iff \bigcap_n P_{x_n} = \emptyset$. By Lusin--Novikov, $G$ is Borel. We now apply \cref{thm:countable-dimensional-dichotomy} to $(G, E)$, and consider the two cases.

	\textbf{Case 1:} There is a lexicographically reducible quasi-order $R$ containing $E$ and a countable Borel $\equiv_R$-local colouring of $G$. Let $F = \mathord{\equiv_R}$, so that $E \subseteq F$ and $F$ is smooth. Let $G'$ be the $\aleph_0$-dimensional graph consisting of $F$-equivalent sequences in $G$, so that a countable Borel $F$-local colouring of $G$ is exactly a countable Borel colouring of $G'$. Thus $\chi_B(G') \leq \aleph_0$.

	By the proof of \cref{prop:smooth-refinement}, there is a smooth Borel equivalence relation $E \subseteq F' \subseteq F$ such that $\bigcap_{x \in C} P_x \neq \emptyset$ for every $F'$-class $C$. Indeed, the proof proceeds in two steps: first the graph $G'$ is constructed (it is denoted by $G$ in that proof), and it is shown that $\chi_B(G') \leq \aleph_0$; second, one constructs $F'$ from a countable Borel colouring of $G'$. We know that $\chi_B(G') \leq \aleph_0$, so the second part of the proof provides the desired equivalence relation $F'$. By \cref{thm:unif-from-smooth-4}, $P$ admits a Borel $F'$-invariant uniformization.

	\textbf{Case 2:} There is a continuous homomorphism $\pi: X_\alpha \to X$ from $(\G_0^\omega, \bbH_0^\omega)$ to $(G, E)$. We will show that (4) holds in \cref{prop:dichotomy-equivalent-conditions}, from which it follows that alternative (2) of \cref{thm:ben-dichotomy} holds. To see this, consider $F = (\pi \times \pi)^{-1}(E)$ and $R = (\pi \times \pi)^{-1}(R')$, where $x R' x' \iff P_x \cap P_{x'} = \emptyset$. Note that $R'$ is Borel by Lusin--Novikov, and hence so is $R$. Also, $\bbH_0^\omega \subseteq F$ and $F \cap R = \emptyset$.

	We claim that $R$ is comeager. To see this, fix $x \in X_\alpha$ and consider
	\[R_x^c = \{x' \in X_\alpha : P_{\pi(x)} \cap P_{\pi(x')} \neq \emptyset\}.\]
	Fix an enumeration $y_n$ of $P_{\pi(x)}$, and let $A_n = \{x' \in X_\alpha : y_n \in P_{\pi(x')}\}$. Then each $A_n$ is $\G_0^\omega$-independent and hence meager; thus so is $R_x^c = \bigcup_n A_n$. Thus $R_x$ is comeager for all $x \in X_\alpha$, and by Kuratowski--Ulam $R$ is comeager.

	One can now recursively construct a continuous homomorphism $f: 2^\omega \to X_\alpha$ from $(\Delta(2^\omega)^c, \E_0^c, \E_0)$ to $((\pi \times \pi)^{-1}(\Delta(X))^c, R, F)$. The construction is exactly the same as in the proof of \cite[Proposition~11]{Mi3}, except that (a) the dense open sets $U_n$ are chosen so that $\bigcap_n U_n \subseteq R$ (this ensures that it is a homomorphism from $\E_0^c$ to $R$), (b) we replace the set $J$ in that proof with the set $T$ used in the construction of $\H_0^\omega$ (this ensures that we get a homomorphism of $\E_0$ into $F$, as $\H_0^\omega \subseteq F$), and (c) we ensure the image of $u$ in the definition of an approximation is $\alpha(k(n))^{k(n)}$ (one can always make an approximation have this property by appending sufficiently many zeros to the end of $u$). Then $\pi \circ f$ satisfies (4).
\end{proof}

We note the following effective version of \cref{thm:ben-dichotomy}:

\begin{thm}\label{thm:effective-ben-dichotomy}
	Let $E$ be a $\Delta^1_1$ equivalence relation on $\N^\N$ and $P \subseteq \N^\N \times \N^\N$ an $E$-invariant $\Delta^1_1$ relation with countable non-empty sections. Then exactly one of the following holds:
	\begin{enumerate}
		\item There is a $\Delta^1_1$ $E$-invariant uniformization,

		\item There is a continuous embedding $\pi_X\colon 2^\N \times \N \to X$ of $\E_0\times I_\N$ into $E$ and a continuous injection $\pi_Y\colon 2^\N \times \N \to Y$ such that for all $x, x' \in 2^\N \times \N$,
		\[\lnot (x \mathrel{\E_0 \times I_\N} x') \implies P_{\pi_X(x)} \cap P_{\pi_X(x')} = \emptyset\]
		and
		\[P_{\pi_X(x)} = \pi_Y([x]_{\E_0 \times I_\N}).\]
	\end{enumerate}
\end{thm}

\begin{proof}[Proof sketch]
	In the proof given above, we showed that the two cases are mutually exclusive and defined a graph $G$ such that 2 holds if there exists a continuous homomorphism from $(\G_0^\omega, \H_0^\omega)$ to $(G, E)$. If $E, P$ are $\Delta^1_1$, then $G$ is $\Delta^1_1$ by the effective Lusin--Novikov Theorem. By \cref{thm:effective-aleph0-dimensional-dichotomy}, we can assume that we are in case 1 of this proof, except that $F$ is $\Delta^1_1$ and there is a countable $\Delta^1_1$ colouring of $G'$. The proofs of \cref{prop:smooth-refinement,thm:unif-from-smooth-4} are sufficiently effective that the smooth equivalence relation $F'$ and $F'$-invariant uniformization of $P$, as constructed in case 1 of this proof, are both $\Delta^1_1$.
\end{proof}

\subsection{Proof of \texorpdfstring{\cref{thm:invariant-lusin-novikov}}{Theorem~\ref*{thm:invariant-lusin-novikov}}}\label{sec:invariant-lusin-novikov}

\begin{proof}[Proof of \cref{thm:invariant-lusin-novikov}]
	Note first that (1) is equivalent to the existence of a smooth Borel equivalence $F \supseteq E$ for which $P$ is $F$-invariant. Indeed, if $F$ exists then (1) holds by \cref{thm:smooth-lusin-novikov}. On the other hand, if (1) holds and is witnessed by $g_n$ then the equivalence relation $x F x' \iff \forall n (g_n(x) = g_n(x'))$ is a smooth Borel equivalence relation containing $E$ for which $P$ is invariant.

	To see that the two alternatives are mutually exclusive, let $F \supseteq E$ be smooth so that $P$ is $F$-invariant and suppose that $\pi_X, \pi_Y$ witness (2). Then there is a comeager $\E_0$-invariant set $C$ that $\pi_X$ maps into a single $F$-class, so $\pi_Y(C)$ is contained in a single $P$-section, a contradiction.

	Now define the graph $x G x' \iff P_x \neq P_{x'}$. This graph is Borel by Lusin--Novikov. We apply the $(\G_0, \bbH_0)$ dichotomy to $(G, E)$ and consider the two alternatives:

	\textbf{Case 1:} There is a smooth $F \supseteq E$ such that $G \cap F$ admits a countable Borel colouring.

	If $A$ is analytic and $G \cap F$-independent then so is $[A]_E$, and by \cref{thm:first-reflection-theorem} any $G \cap F$-independent analytic set is contained in an $G \cap F$-independent Borel set. We may therefore construct an increasing sequence $B_n$ of $G \cap F$-independent Borel sets with $A \subseteq B_0$ and $[B_n]_E \subseteq B_{n+1}$. Then $B = \bigcup_n B_n$ is Borel, $E$-invariant, $G \cap F$-independent, and contains $A$.

	We may therefore fix a sequence $B_n$ of Borel, $E$-invariant, $G \cap F$-independent sets covering $X$. If we let $x F' x' \iff x F x' \And \forall n (x \in B_n \iff x' \in B_n)$, then $E \subseteq F' \subseteq F$, $F'$ is a smooth Borel equivalence relation, and $G \cap F' = \emptyset$. Thus (1) holds.

	\textbf{Case 2:} There is a continuous homomorphism $\varphi: 2^\N \to X$ from $(\G_0, \bbH_0)$ to $(G, E)$. Define $R(x, y) \iff P(\varphi(x), y)$, and let
	\[Q(x, y) \iff R(x, y) \And \forall^* x' \lnot R(x', y),\]
	where $\forall^* x A(x)$ means $A$ is comeager for $A \subseteq 2^\N$. Let $A = \proj_{2^\N}(Q)$ and $x S x' \iff Q_x \cap Q_{x'} \neq \emptyset$. Then $R$ is Borel with countable sections, and it follows by Lusin--Novikov that $Q, A, S$ are Borel as well. Additionally, $R$ is $F$-invariant, where $F = (\varphi \times \varphi)^{-1}(E)$, and hence so are $Q, A, S$.

	\begin{claim}
		$A$ is comeager and $S$ is meager.
	\end{claim}

	\begin{claimproof}
		We show first that $S$ is meager. By Kuratowski--Ulam, it suffices to show that $S$ has meager sections. Now if $x \notin A$ then $Q_x = \emptyset$, so $S_x = \emptyset$ as well. On the other hand, if $x \in A$ and $y \in Q_x$ then $R^y = \{x' : R(x', y)\}$ is meager, so $\{x' : \exists y \in Q_x (R(x', y))\}$ is meager. But this set contains $S_x$, so $S_x$ is meager as well.

		Note that $\bbH_0 \subseteq F$, so by \cite[Lemma~4]{Mi3} any non-meager $F$-invariant Borel set is comeager. Define now $x B x' \iff R_x \subseteq R_{x'}$, and note that $B$ is Borel by Lusin--Novikov. We claim that $x \in A$ iff $B_x$ is meager. To see this, suppose first that $x \in A$. Then $R^y$ is meager for some $y \in R_x$, and hence so is $B_x = \bigcap_{y \in R_x} R^y$. Conversely, if $B_x$ is meager then $R^y$ is not comeager for some $y \in R_x$. But $R^y$ is $F$-invariant, hence meager, so $y \in Q_x$ and $x \in A$. Thus, by Kuratowski--Ulam, in order to show that $A$ is comeager it suffices to show that $B$ is meager.

		Suppose towards a contradiction that $B$ is non-meager. Let $C$ be the set of all $x$ for which $B_x$ is non-meager. Then $C$ is Borel and $F$-invariant, and by Kuratowski--Ulam it is non-meager, so $C$ is comeager. Similarly, if $x \in C$ then $B_x$ is Borel, $F$-invariant and non-meager, so $B_x$ is comeager for $x \in C$. It follows by Kuratowski--Ulam that $B$ is comeager.

		Let now $x B' x' \iff x B x' \And x' B x \iff R_x = R_{x'}$. As $B$ is comeager so too is $B'$, and hence there is some $x$ for which $D = (B')_x$ is comeager. For all $x', x'' \in D$ we have $R_{x'} = R_x = R_{x''}$, so $x' \cancel{\G_0} x''$ and $D$ is $\G_0$-independent. But every $\G_0$-independent Borel set is meager by \cite[Proposition~6.2]{KST}, a contradiction.
	\end{claimproof}

	By \cite[Proposition~11]{Mi3}, we may find a continuous homomorphism $\psi: 2^\N \to A$ from $(\E_0, \E_0^c)$ to $(F, S^c)$ such that $\varphi \circ \psi: 2^\N \to X$ is injective. Now the set $Q'(x, y) \iff Q(\psi(x), y)$ has countable sections, so by Lusin--Novikov it admits a Borel uniformization $g$.

	Since $\psi$ is a homomorphism from $\E_0^c$ to $S^c$, $g$ is countable-to-one, so by Lusin--Novikov there is a Borel non-meager set $B$ on which $g$ is injective. By \cite[8.38]{CDST} we may assume that $\res{g}{B}$ is continuous as well, and by the proof of \cref{claim:embed-E0-into-itself} we may fix a continuous embedding $\tau: 2^\N \to B$ of $\E_0$ into $\res{\E_0}{B}$. Then $\pi_X = \varphi \circ \psi \circ \tau$, $\pi_Y = g \circ \tau$ satisfy (2) of \cref{thm:invariant-lusin-novikov}.
\end{proof}

\begin{rmk}
	This proof actually shows that in case (2), we can take $\pi_X, \pi_Y$ so that additionally $\pi_Y(x) \in P_{\pi_X(x')} \iff x \E_0 x'$.
\end{rmk}

Using \cref{thm:effective-aleph0-dimensional-dichotomy}, this proof gives the following effective analogue of \cref{thm:invariant-lusin-novikov}:

\begin{thm}
	Let $E$ be a $\Delta^1_1$ equivalence relation on $X$ and $P \subseteq \N^\N \times \N^\N$ an $E$-invariant $\Delta^1_1$ relation with countable non-empty sections. Then exactly one of the following holds:

	(1) There is a uniformly $\Delta^1_1$ sequence $g_n: X \to Y$ of $E$-invariant uniformizations with $P = \bigcup_n \graph(g_n)$,

	(2) There is a continuous embedding $\pi_X: 2^\N \to X$ of $\E_0$ into $E$ and a continuous injection $\pi_Y: 2^\N \to Y$ such that for all $x \in 2^\N$, $P(\pi_X(x), \pi_Y(x))$.
\end{thm}

\subsection{Proofs of \texorpdfstring{\cref{prop:complexity-countable-sections,thm:complexity-large-sections}}{Proposition~\ref*{prop:complexity-countable-sections}~and~Theorem~\ref*{thm:complexity-large-sections}}}\label{sec:hardness-proofs}

Let us fix a parametrization of the Borel relations on $\N^\N$, as in \cite[Section~5]{AK} (see also \cite[Section~3H]{Mo}). This consists of a set $\bbD \subseteq 2^\N$ and two sets $S, P \subseteq (\N^\N)^3$ such that
\begin{enumerate}[label=(\roman*)]
	\item $\bbD$ is $\Pi^1_1$, $S$ is $\Sigma^1_1$ and $P$ is $\Pi^1_1$;
	\item for $d \in \bbD$, $S_d = P_d$ (and in particular this set is Borel);
	\item for every Borel set $B \subseteq (\N^\N)^2$ there is some $d \in \bbD$ satisfying $B = S_d = P_d$; and
	\item if $X$ a Polish space and $B \subseteq X \times (\N^\N)^2$ is Borel, then there is a Borel function $p: X \to 2^\N$ so that $p(X) \subseteq \bbD$ and $B_x = S_{p(x)} = P_{p(x)}$ for all $x \in X$.
\end{enumerate}
When $d \in \bbD$, we abuse notation and write $\bbD_d$ for the set $S_d = P_d$.

Define
\[\calP = \{(d, e) : \text{$\bbD_d$ is an equivalence relation on $\N^\N$ and $\bbD_e$ is $\bbD_d$-invariant}\},\]
and let $\calP^{unif}$ denote the set of pairs $(d, e) \in \calP$ for which $\bbD_e$ admits a Borel $\bbD_d$-invariant uniformization. More generally, for any set $A$ of properties of sets $P \subseteq \N^\N \times \N^\N$, let $\calP_A$ (resp. $\calP_A^{unif}$) denote the set of pairs $(d, e)$ in $\calP$ (resp. $\calP^{unif}$) such that $\bbD_e$ satisfies all of the properties in $A$. It is straightforward to check that $\calP$ is $\bm{\Pi^1_1}$, and that $\calP_A^{unif}$ is $\bm{\Sigma^1_2}$ whenever $\calP_A$ is $\bm{\Sigma^1_2}$. (Note that ``$\bbD_f$ is a function'' is a $\bm{\Pi^1_1}$ property of $f$ by \cref{thm:Ksigma-delta11-point,thm:restricted-quantification}; see also the paragraph before the proof of \cref{thm:turbulence-example}.)

We are interested in properties asserting that $\bbD_e$, or its sections, are (a) comeager, non-meager, $\mu$-positive or $\mu$-conull, where $\mu$ varies over probability Borel measures on $\N^\N$; (b) nonempty and either countable or $K_\sigma$; or (c) $F_\sigma$ or $G_\delta$. For all sets $A$ of these properties, $\calP_A$ is $\bm{\Pi^1_1}$ and $\calP^{unif}_A$ is $\bm{\Sigma^1_2}$. (For properties in (a) this follows from \cite[32.4]{CDST}; for (b) we use \cite[29.19]{CDST}, \cite[4F.10, 4F.18]{Mo} and \cref{thm:Ksigma-delta11-point,thm:restricted-quantification}; and for (c) we apply Louveau's Theorem \cite[Theorem~A]{Louveau-separation}, i.e., the effective generalization of \cref{thm:Hurewicz}, and \cref{thm:restricted-quantification} (see also \cite[Lemma~4]{Louveau-separation}).)

Let $\calP_{ctble}$ (resp. $\calP_{ctble}^{unif}$) denote $\calP_A$ (resp. $\calP_A^{unif}$) for $A$ consisting of the property that $P$ has countable nonempty sections. By \cref{thm:effective-ben-dichotomy}, we can bound the complexity of $\calP_{ctble}^{unif}$ further:

\begin{prop}[\cref{prop:complexity-countable-sections}]
	The set $\calP_{ctble}^{unif}$ is $\bm{\Pi^1_1}$.
\end{prop}

\begin{proof}
	We note that $(d, e) \in \calP^{unif}_{ctble}$ iff $(d, e) \in \calP_{ctble}$ and there exists a $\Delta^1_1(d, e)$ function $f$ which is a $\bbD_d$-invariant uniformization of $\bbD_e$. Indeed, if there was no such function $f$ which is $\Delta^1_1(d, e)$, then 2 would hold in \cref{thm:effective-ben-dichotomy} (as the proof of this theorem relativizes). Then by \cref{thm:ben-dichotomy} there would be no such function $f$ which is Borel. The assertion that a $\Delta^1_1(d, e)$ function $f$ is a $\bbD_d$-invariant uniformization of $\bbD_e$ is $\Pi^1_1(d, e)$, so $\calP^{unif}_{ctble}$ is $\Pi^1_1$ by \cref{thm:restricted-quantification}.
\end{proof}

Recall that a set $B$ in a Polish space $X$ is called \tb{$\bm{\Sigma^1_2}$-complete} if it is $\bm{\Sigma^1_2}$, and for all zero-dimensional Polish spaces $Y$ and $\bm{\Sigma^1_2}$ sets $C \subseteq Y$ there is a continuous function $f: Y \to X$ such that $C = f^{-1}(B)$. Note that by \cite[Theorem~2]{Sa} one could equivalently take $f$ to be Borel in this definition (see also \cite{P}).

The following computes the exact complexity of the sets $\calP_A^{unif}$, when $A$ asserts that $\bbD_e$ has ``large'' sections.

\begin{thm}[\cref{thm:complexity-large-sections}]
	The set $\calP_A^{unif}$ is $\bm{\Sigma^1_2}$-complete, where $A$ is one of the following sets of properties of $P \subseteq \N^\N \times \N^\N$:
	\begin{enumerate}
		\item $P$ has non-meager sections;
		\item $P$ has non-meager $G_\delta$ sections;
		\item $P$ has non-meager sections and is $G_\delta$;
		\item $P$ has $\mu$-positive sections for some probability Borel measure $\mu$ on $\N^\N$;
		\item $P$ has $\mu$-positive $F_\sigma$ sections for some probability Borel measure $\mu$ on $\N^\N$;
		\item $P$ has $\mu$-positive sections for some probability Borel measure $\mu$ on $\N^\N$ and is $F_\sigma$.
	\end{enumerate}
	The same holds for comeager instead of non-meager, and $\mu$-conull instead of $\mu$-positive.

	In fact, there is a hyperfinite Borel equivalence relation $E$ with code $d \in \bbD$ such that for all such $A$ above, the set of $e \in \bbD$ such that $(d, e) \in \calP_A^{unif}$ is $\bm{\Sigma^1_2}$-complete.
\end{thm}

\begin{proof}
	We will show this first when $A$ asserts that $P$ is $G_\delta$ and has comeager sections. Since $\N^\N$ is Borel isomorphic to $\N^\N \times 2^\N$, we may assume that $\bbD_d$ is instead an equivalence relation on $\N^\N \times 2^\N$ and that $\bbD_e \subseteq (\N^\N \times 2^\N) \times \N^\N$.

	Let $E$ be the hyperfinite Borel equivalence relation on $\N^\N \times 2^\N$ given by
	\[(x, y) E (x', y') \iff x = x' \And y \E_0 y',\]
	fix a code $d \in \bbD$ for $E$, and let $\calP_A^{unif}(E)$ denote the set of all $e \in \bbD$ so that $(d, e) \in \calP_A^{unif}$. We will show that $\calP_A^{unif}(E)$ is $\bm{\Sigma^1_2}$-complete.

	Let now $T$ be a tree on $\N \times \N$ (cf. \cite[2.C]{CDST}). Each such tree $T$ defines a closed subset $[T] \subseteq \N^\N \times \N^\N$ given by
	\[[T] = \{(x, y) \in \N^\N \times \N^\N : \forall n \, ((\res{x}{n}, \res{y}{n}) \in T)\}.\]
	We say $[T]$ admits a \textbf{full Borel uniformization} if there is a Borel map $f: \N^\N \to \N^\N$ so that $(x, f(x)) \in [T]$ for all $x \in \N^\N$, and we denote by $\FBU$ the set of trees on $\N \times \N$ which admit full Borel uniformizations.

	In \cref{subsec:non-smooth-case}, we constructed a $G_\delta$ set $Q \subseteq 2^\N \times 2^\N$ which is $\E_0$-invariant, has comeager sections, and satisfies
	\[\bigcap_{x \in D} Q_x = \emptyset\]
	whenever $D \subseteq 2^\N$ is $\mu$-positive, where $\mu$ is the uniform product measure on $2^\N$. By considering $\N^\N$ as a co-countable set in $2^\N$, we may take $Q$ satisfying $Q \subseteq 2^\N \times \N^\N$.

	Given a tree $T$ on $\N\times \N$, define $Q_T \subseteq (\N^\N \times 2^\N) \times \N^\N$ by
	\[Q_T(x, y, z) \iff Q(y, z) \lor (x, z) \in [T].\]
	Note that $Q_T$ is $G_\delta$, $E$-invariant, and has comeager sections.

	\begin{claim}\label{claim:reduction-FBU}
		$[T]$ admits a full Borel uniformization iff $Q_T$ admits a Borel $E$-invariant uniformization.
	\end{claim}

	\begin{claimproof}
		If $f$ is a full Borel uniformization of $[T]$, then $g(x, y) = f(x)$ is an $E$-invariant Borel uniformization of $Q_T$. Conversely, suppose $g$ were an $E$-invariant Borel uniformization of $Q_T$. For $x \in \N^\N$, let $g_x(y) = g(x, y)$. Then $g_x: 2^\N \to \N^\N$ is $\E_0$-invariant, hence constant on a $\mu$-conull set $D \subseteq 2^\N$. Since
		\[\bigcap_{y \in D} P_y = \emptyset,\]
		we cannot have $P(y, g_x(y))$ for all $y \in D$, and so $(x, g_x(y)) \in [T]$ for all $y \in D$. Thus
		\[f(x) = z \iff \forall^*_\mu y (g(x, y) = z)\]
		is a full Borel uniformization of $[T]$ (cf. \cite[17.25, 17.26]{CDST} and the paragraphs following it).
	\end{claimproof}

	By identifying trees on $\N \times \N$ with their characteristic functions, we can view the space of trees as a closed subset of $2^\N$. The set $B$ given by
	\[B(T, x, y, z) \iff T \text{ is a tree and } Q_T(x, y, z)\]
	is clearly Borel, so there is a Borel map $p$ such that for each tree $T$, $p(T) \in \bbD$ and $\bbD_{p(T)} = Q_T$. It follows by \cref{claim:reduction-FBU} that $\FBU = p^{-1}(\calP_A^{unif}(E))$. By \cite[Lemma~5.3]{AK}, the set $\FBU$ is $\bm{\Sigma^1_2}$-complete, and hence so is $\calP_A^{unif}(E)$.

	The cases 1--3 follow from this as well. For 4--6, simply replace $Q$ in the above proof with $P = 2^\N \times \N^\N \setminus Q$. Then $P$ is $F_\sigma$, $\E_0$-invariant, has $\mu$-conull sections, and satisfies
	\[\bigcap_{x \in C} P_x = \emptyset\]
	whenever $C \subseteq 2^\N$ is non-meager (c.f. \cref{subsec:non-smooth-case}). The proof is then identical, after replacing the measure quantifier $\forall^*_\mu y$ in the proof of \cref{claim:reduction-FBU} by the category quantifier $\forall^* y$ (c.f. \cite[8.J, 16.A]{CDST}).
\end{proof}

\begin{rmk}
	We do not know the complexity of $\calP_A^{unif}$ when $A$ asserts that $P$ is $G_\delta$ and has comeager, $\mu$-conull sections for a probability Borel measure $\mu$. By the proof of \cref{thm:ramsey-example}, there is an $\E_0$-invariant $G_\delta$ set $R \subseteq [\N]^{\aleph_0} \times \N^\N$ with comeager, $\mu$-conull sections satisfying
	\[\bigcap_{x \in C} P_x = \emptyset\]
	for all Ramsey-positive sets $C \subseteq [\N]^{\aleph_0}$. One can define $P_T$ for a tree $T$ on $\N \times \N$ as in the proof of \cref{thm:complexity-large-sections}, however the ``if'' direction of our proof of \cref{claim:reduction-FBU} no longer works (cf. \cite{Sa}).
\end{rmk}

\subsection{Proof of \texorpdfstring{\cref{prop:K-sigma-failure}}{Proposition~\ref*{prop:K-sigma-failure}}}\label{sec:K-sigma-failure}

By \cite[18.17]{CDST}, there is a $G_\delta$ set $R \subseteq \N^\N \times 2^\N$ with $\operatorname{proj}_{\N^\N}(R) = \N^\N$ which does not admit a Borel uniformization. Write $R = \bigcap_n Q_n$, $Q_n \subseteq \N^\N \times 2^\N$ open, and define $P$ by
\[P(n, x, y) \iff Q_n(x, y).\]
Let $(n, x) F (m, x') \iff x = x'$. Then $F$ is a smooth countable Borel equivalence relation, $P$ is open, and if $C = [(n, x)]_F$ is an $F$-class then
\[\bigcap_{u \in C} P_u = \bigcap_n P_{(n, x)} = \bigcap_n (Q_n)_x = R_x \neq \emptyset.\]

Suppose now towards a contradiction that $g: \N \times \N^\N \to 2^\N$ is an $F$-invariant uniformization of $P$. Define $f: \N^\N \to 2^\N$ by $f(x) = g(0, x)$. Then $f(x) = g(0, x) = g(n, x) \in P_{(n, x)}$ for all $n$, so $f(x) \in \bigcap_n P_{(n, x)} = R_x$, a contradiction.

\section{On \texorpdfstring{\cref{conj:countable-uniformization}}{Conjecture~\ref*{conj:countable-uniformization}}}\label{sec:countable-uniformization}
Concerning \cref{conj:countable-uniformization}, we first note the following analog of \cref{lem:uniformization-closed-reducibility}.

\begin{lem}\label{lem:countable-uniformization-closed-reducibility}
Let $E,F$ be  Borel equivalence relations on Polish spaces $X,X'$, resp., such that $E\leq_B E'$. If $E$ fails (b) (resp., (c), (d)), so does $E'$.
\end{lem}

The proof is identical to that of \cref{lem:uniformization-closed-reducibility}.

Note now that any countable Borel equivalence relation $E$ trivially satisfies (b), (c), and (d), so by \cref{lem:countable-uniformization-closed-reducibility}, in \cref{conj:countable-uniformization}, (a) implies (b), (c), and (d). To verify then \cref{conj:countable-uniformization}, one needs to show that if $E$ is not reducible to countable, then (b), (c), and (d) fail.

The following are examples of Borel equivalence relations that are known not to be reducible to countable (see e.g. \cite[8.1.6, 8.5.2, 10.4.3]{Gao-IDST}):
\begin{enumerate}[label=(\arabic*)]
	\item The equivalence relation $\E_1$ on $(2^\N)^\N$ given by
	\[x\E_1 y \iff \exists m \forall n \geq m (x_n = y_n).\]
	\item Any Borel equivalence relation $F$ induced by a turbulent continuous action of a Polish group on a Polish space (c.f. \cite[Definition~2.10]{HK}).
	\item The equivalence relation $\E_0^\N \leq_B E$ on $(2^\N)^\N$ given by
	\[x\E_0^\N y \iff \forall n(x_n \E_0 y_n).\]
\end{enumerate}
It is an open problem (see \cite[end of Section 6]{HK}) whether every Borel equivalence relation that is not reducible to countable lies above one of these relations with respect to Borel reducibility. It is therefore interesting to show that (b), (c), and (d) fail for $\E_1$, $F$ as in (2) above, and $\E_0^\N$. Below are some partial results.

\begin{prop}
Let $E$ be a Borel equivalence relation which is not reducible to countable but is Borel reducible to a Borel equivalence relation $F$ with $K_\sigma$ classes. Then $E$ fails (d). In particular, $\E_1$ and $\E_2$ fail (d), where $\E_2$ is the following equivalence relation on $2^\N$:
\[x \E_2 y \iff \sum_{n : x_n \neq y_n} \frac{1}{n+1} < \infty.\]
\end{prop}

\begin{proof}
Suppose $E,F$ live on the Polish spaces $X,Y$, resp., and let $g\colon X\to Y$ be a Borel reduction of $E$ to $F$. Define $P\subseteq X\times Y$ as follows:
\[(x,y) \in P \iff g(x) F y.\]
Clearly $P$ is $E$-invariant and has $K_\sigma$ sections. Suppose then that $P$ admitted a Borel $E$-invariant countable uniformization $f\colon X\to Y^\N$. Then define $h\colon X\to Y$ by $h(x) = f(x)_0$. Then by \cite[Proposition~4.7]{CBER}, $h$ shows that $E$ is reducible to countable, a contradiction.
\end{proof}

Concerning (b) and (c) for $\E_1$, the following are possible examples for their failure.

\begin{prob}
	Let $X = (2^\N)^\N$, $Y = 2^\N$. Identifying $2^\N$ with the group $(\Z/2\Z)^\N$ with Haar measure $\mu$, let $B \subseteq Y$ be a dense $G_\delta$ $\mu$-null set and define
	\[(x, y) \in Q \iff \forall x' \E_0 x (x' + y \in B)\]
	(this is the same set $Q$ defined in \cref{subsec:non-smooth-case}). Define now $P \subseteq X \times Y$ as follows:
	\[(x, y) \in P \iff \exists m \forall n \geq m (y \notin Q_{x_n}).\]
	Then $P$ is $\E_1$-invariant and each section $P_x$ is $\mu$-conull. Moreover, $P$ has meager sections, so if we take $P' = X \times Y \setminus P$ then $P'$ is $\E_1$-invariant and each section $P'_x$ is comeager. Is there a Borel $\E_1$-invariant countable uniformization of $P$, or of $P'$?
\end{prob}

In a previous draft of this paper, we looked at the set $R \subseteq (2^\N)^\N \times 2^\N$ given by
\[(x, y) \in R \iff \exists m \forall n \geq m (y \neq x_n).\]
This set is $\E_1$-invariant and has co-countable sections. The anonymous referee pointed out that this set admits a Borel $\E_1$-invariant countable uniformization: For each $x \in (2^\N)^\N$, let $f(x) \in 2^\N$ denote the standard diagonalization against $x$, namely $f(x)(n) = 1 - x(n)(n)$, and let $g(x)$ enumerate the $\E_0$-class of $f(x)$. Then $g$ is a Borel $\E_1$-invariant countable uniformization of $R$.

One has, however, the following Borel anti-diagonalization theorem for $\E_1$, which shows in particular that $R$ does not admit a Borel $\E_1$ invariant uniformization.

\begin{prop}\label{prop:E_1-anti-diagonalization}
Let $f: (2^\N)^\N \to 2^\N$ be a Borel function such that $x \E_1 y \implies f(x) = f(y)$. Then there is $x\in (2^\N)^\N$ such that for infinitely many $n$, $f(x) = x_n$.

In particular, the set $R \subseteq (2^\N)^\N \times 2^\N$ defined above does not admit a Borel $\E_1$-invariant uniformization.
\end{prop}

\begin{proof}
For any nonempty countable set $S\subseteq 2^\N$ consider the product space $S^\N$ with the product topology, where $S$ is taken to be discrete. Denote by $\E_0 (S)$ the equivalence relation on $S^\N$ given by $x\E_0 (S)  y \iff \exists m \forall n \geq m (x_n = y_n)$. This is generically ergodic and for $x,y \in S^\N$ we have that $x\E_0 (S)  y \implies f(x) = f(y)$, so there is (unique) $x_S\in 2^\N$ such that $f(x) = x_S$, for comeager many $x\in S^\N$. Clearly $x_S$ can be computed in a Borel way given any $x\in (2^\N)^\N$ with $S= \{ x_n \colon n \in \N\}$, i.e., we have  a Borel function $F\colon (2^\N)^\N \to 2^\N$ such that
\[\{ x_n\colon n\in \N \} = \{y_n \colon n\in \N\}  =S \implies F((x_n) ) = F((y_n)) = x_S.\]
We now use the following Borel anti-diagonalization theorem of H. Friedman, see \cite[Theorem 2, page 23]{S}:

\begin{thm}[H. Friedman]\label{thm:friedman}
Let $E$ be a Borel (even analytic) equivalence relation on a Polish space $X$. Let $F\colon X^\N \to X$ be a Borel function such that
 \[\{ [x_n]_E\colon n\in \N \} = \{[y_n]_E\colon n\in \N\} \implies F((x_n) ) \  E  \ F((y_n)).\]
Then there is $x\in X^\N$ and $i\in \N$ such that $F(x) E x_i$.
\end{thm}

Applying this to $E$ being the equality relation on $2^\N$ and $F$ as above, we conclude that for some $S$, we have that $x_S\in S$. Then for comeager many $x\in S^\N$ we have that $x_n = x_S$, for infinitely many $n$.

Let now $R \subseteq (2^\N)^\N \times 2^\N$ be defined by
\[(x, y) \in R \iff \exists m \forall n \geq m (y \neq x_n).\]
If $f$ were a Borel $\E_1$-invariant uniformization of $R$, then there would be some $x \in (2^\N)^\N$ with $f(x) = x_n$ for infinitely many $n$. But then $(x, f(x)) \notin R$, a contradiction.
\end{proof}

In response to a question by Andrew Marks, we note the following version of \cref{prop:E_1-anti-diagonalization} for $\E_1$ restricted to injective sequences. Below $[2^\N]^\N$ is the Borel subset of $(2^\N)^\N$ consisting of injective sequences and $x\leq_T y$ means that $x$ is recursive in $y$.

\begin{prop}
Let $g\colon [2^\N]^\N \to 2^\N$ be a Borel function such that $x \E_1 y \implies g(x) = g(y)$. Then there is $y\in [2^\N]^\N$ such that for all $n$, $g(y) \leq_T y_n$.
\end{prop}

\begin{proof}
Fix a recursive bijection  $x\mapsto \langle x \rangle$ from $(2^\N)^\N$ to $2^\N$ and for each $i\in \N$ let $\bar{i}\in 2^\N$ be the characteristic function of $\{i\}$. Then for each $x\in (2^\N)^\N$ and $i\in \N$, put
\[
\bar{x}^i = \langle \bar{i}, x_i, x_{i+1}, \dots \rangle \in 2^\N
\]
and
\[
x' = \langle \bar{x}^0, \bar{x}^1, \dots\rangle\in [2^\N]^\N.
 \]
Note that $x \E_1 y \implies x' \E_1 y'$. Finally define $f\colon (2^\N)^\N \to 2^\N$ by $f(x) = g(x')$. Then by \cref{prop:E_1-anti-diagonalization}, there is $x\in (2^\N)^\N$ such that for infinitely many $n$ we have that $f(x) = x_n$. Let $y = x'$.

If $n$ is such that $f(x) = g(y) = x_n$, then as $x_n \leq_T \bar{x}^k = y_k, \forall k\leq n$, we have that  $g(y) \leq_T y_k, \forall k\leq n$. Since this happens for infinitely many $n$, we have that $g(y) \leq_T y_n$, for all $n$.
\end{proof}

We do not know anything about $\E_0^\N$ but if we let $\E_{ctble}$ be the equivalence relation $\E_{ctble}^{2^\N}$ (so that $\E_0^\N <_B \E_{ctble}$), we have:

\begin{prop}\label{prop:Ectbl}
$\E_{ctble}$ fails (b) and (c).
\end{prop}

\begin{proof}
We will prove that $\E_{ctble}$ fails (b), the proof that it also fails (c) being similar. Let $X = (2^\N)^\N, Y = 2^\N$, let $\mu$ be the usual product measure on $Y$ and put $E= \E_{ctble}$. Define $P\subseteq X\times Y$ by
\[
(x, y) \in P \iff y\notin \{x_n\colon n \in \N\}.
\]
Clearly $\mu (P_x) =1$ and $P$ is $E$-invariant. Assume now, towards a contradiction, that there is a Borel function $f\colon X \to Y^\N$ such that $\forall x\in X \forall n\in \N ( (x, f(x)_n) \in P)$ and $x_1 E x_2 \implies \{f(x_1)_n\colon n\in \N\} = \{f(x_2)_n\colon n\in \N\}$. Then
\[
 \forall x\in X\big( \{f(x)_n \colon n\in \N\} \cap \{x_n \colon n\in \N\} =\emptyset\big).
\]
Define $F\colon X^\N \to Y^\N$ as follows: Fix a bijection $(i,j) \mapsto \langle i, j \rangle$ from $\N^2$ to $\N$ and for $n\in \N$ put $n = \langle n_0, n_1 \rangle$. Given $x\in X^\N$, define $x'\in X$ by $x'_n = (x_{n_0})_{n_1} $. Then let $F(x) = f(x')$. First notice that for $x = (x_n),y = (y_n)\in X^ \N$,
 \[
\{ [x_n]_E\colon n\in \N \} = \{[y_n]_E\colon n\in \N\} \implies x'   E  y' \implies  F(x )   E   F(y).
\]
Thus by \cref{thm:friedman}, there is some $x\in X^\N$ and $i\in \N$ such that $F(x) E x_i$, i.e., $f(x') E x_i$ or $\{f(x')_n\colon n\in \N\} = \{(x_i)_n \colon n\in \N\} = \{ x'_{\langle i, n\rangle}\colon n\in \N\}$. Thus $\{f(x')_n \colon n\in \N\} \cap \{x'_n \colon n\in \N\} \not=\emptyset$, a contradiction.
\end{proof}

We do not know if $\E_{ctble}$ fails (d). We also do not know anything about equivalence relations induced by turbulent continuous actions of Polish groups on Polish spaces.

Finally, we note that by the dichotomy theorem of Hjorth concerning reducibility to countable (see \cite{H} or \cite[Theorem~4.8]{CBER}), in order to prove \cref{conj:countable-uniformization} for Borel equivalence relations induced by Borel actions of Polish groups, it would be sufficient to prove it for Borel equivalence relations induced by stormy such actions.

\printbibliography

\noindent Department of Mathematics

\noindent California Institute of Technology

\noindent Pasadena, CA 91125

\noindent\textsf{kechris@caltech.edu}

\noindent\textsf{mwolman@caltech.edu}

\end{document}